\documentclass[a4paper,11pt,reqno]{amsart}

\usepackage[left=1in,right=1in,top=1in,bottom=1in]{geometry}
\usepackage{enumitem}
\usepackage{amssymb,amsmath,latexsym,amsfonts,amsbsy,amsthm,mathtools,color}
\usepackage{float}
\usepackage[colorlinks=true,citecolor=blue,filecolor=blue,linkcolor=blue,urlcolor=blue]{hyperref}

\usepackage{esint}

\newcommand{\om}{\Omega}

\newcommand{\va}{\varepsilon}

\newcommand{\be}{\begin{equation}}
\newcommand{\ee}{\end{equation}}

\newcommand{\Z}{\mathbb{Z}}

\newcommand{\R}{\mathbb{R}}

\newcommand{\cT}{\mathcal{T}}

\newcommand{\ep}{\epsilon}

\newcommand{\op}{\operatorname}

\newcommand{\abs}[1]{\left\lvert#1\right\rvert}
\newcommand{\norm}[1]{\left\lVert#1\right\rVert}

\newcommand{\dd}{\,\mathrm{d}}

\DeclareMathOperator{\supp}{supp}

\DeclareMathOperator{\tr}{tr}

\DeclareMathOperator{\Sing}{Sing}

\DeclareMathOperator{\loc}{loc}

\setlist[enumerate,1]{label=$(\theenumi)$,ref=\theenumi}
\numberwithin{equation}{section}

\theoremstyle{plain}
\newtheorem{thm}{Theorem}[section]
\newtheorem{cor}[thm]{Corollary}

\newtheorem{lem}[thm]{Lemma}
\newtheorem{prop}[thm]{Proposition}

\theoremstyle{definition}

\newtheorem{example}[thm]{Example}
\newtheorem{q}[thm]{Question}

\theoremstyle{remark}
\newtheorem{rem}[thm]{Remark}

\title[Small-energy regularity for the Liouville equation]{Small-energy regularity for stationary weak solutions of the Liouville equation}
\author{Haotong Fu}
\address{School of Mathematical Sciences, Peking University, Beijing 100871, China}
\email{2301110012@pku.edu.cn}

\author{Kelei Wang}
\address{School of Mathematics and Statistics, Wuhan University, Wuhan 430072, China}
\email{wangkelei@whu.edu.cn}

\author{Wei Wang}
\address{School of Mathematical Sciences, Peking University, Beijing 100871, China}
\email{wwmath166@outlook.com}
\email{2201110024@stu.pku.edu.cn}

\author{Ke Wu}
\address{School of Mathematics, Yunnan Normal University, Kunming, 650500, China}
\email{m18629096093@163.com}

\date{}

\begin{document}
\begin{abstract}
We establish small-energy regularity for stationary weak solutions of the Liouville equation $ -\Delta u=e^u $ in every dimension. It solves a problem left open in Da Lio [Commun.
Partial Differ. Equations, 2008] and posed explicitly again by Da Lio and Hyder in [Commun.
Contemp. Math. 2025]. The proof relies on the analysis of concentration phenomena  for almost stationary fields with structure constraints.
\end{abstract}
\maketitle

\section{Introduction}

\subsection{Background and main results}
In this paper, we consider weak solution $ u $ of the equation
\be
-\Delta u = e^u\quad\text{in }\Omega\label{equationLiouville},
\ee
where $\Omega\subset\subset\R^n$ is a open subset. Here
a function $u\in H^1_{\rm loc}(\Omega)$ with $e^u\in L^1_{\rm loc}(\Omega)$ is said to be a weak solution of \eqref{equationLiouville} if
\[
 \int_\Omega\nabla u\cdot\nabla\zeta\dd x
 =\int_\Omega e^u\zeta\dd x,\quad \forall\zeta\in C^\infty_0(\Omega).
\]
Moreover, $ u $ is called stationary if, in addition,
\begin{equation}\label{eq:stationarity-definition}
 \op{div}\left(\nabla u\otimes\nabla u-\frac{1}{2}\abs{\nabla u}^2I+e^uI\right)=0
 \quad\hbox{in }\mathcal D'(\Omega).
\end{equation}
Indeed, this implies $u$ is a critical point of the energy functional 
\[
\int_{\Omega}\left(\frac{1}{2}\abs{\nabla u}^2-e^u\right) 
\]
with respect to the domain-variation. 

In the two dimensional case, Br{\'e}zis and Merle \cite{BM91} proved that every weak solution of \eqref{equationLiouville} is smooth. When $ n\geq 3 $, weak solutions of \eqref{equationLiouville} is not necessarily smooth, even for stationary solutions. Let us consider the energy
\[
 \mathcal E_u(x,r)=r^{2-n}\int_{B_r(x)}\left(\abs{\nabla u}^2+e^u\right)\dd y.
\]
Inspired by \cite{Lin01}, Da Lio  \cite{DaLio} (see also the improvement in Da Lio-Hyder \cite[Theorem 1.2]{DaLioHyder}) showed that when $ n = 3 $, there is an absolute constant $ \va_0>0 $ such that if a stationary solution $ u $ of \eqref{equationLiouville} satisfies $ \mathcal{E}_u(x,2r)\leq\va_0 $, then $ u $ is smooth in $ B_r(x) $. This is called the $\va$-regularity result for stationary solutions. This $\va$-regularity implies that the Hausdorff dimension of the singular set $ \Sing(u) $ is no more than $ 1 $ in dimension three. Here the singular set of $ u $  contains all points in $ \om $, of which $ u $ is not smooth in every neighborhood.  Later in \cite{DaLioHyder}, Da Lio and Hyder established a blow-up analysis theory for stationary solutions of \eqref{equationLiouville}. In the final paragraph of Section~1, they posed the problem about the higher-dimensional blow-up analysis theory. A key problem is whether the $\va$-regularity result holds in any dimension.

Our main theorem  address this partial-regularity question. More precisely, we prove the following result.
\begin{thm}[Small-energy regularity]\label{thm:weak-epsilon}
For every $n\geq3$, there are $\varepsilon_n>0$ and $C_n<+\infty$ such that every stationary weak solution $u\in H^1(B_2)$, with $e^u\in L^1(B_2)$ and $\mathcal E_u(0,2)\leq\varepsilon_n$, is smooth on $B_{\frac{1}{2}}$ and satisfies
\[
 \sup_{B_{\frac{1}{4}}}e^u\leq C_n\int_{B_2}e^u\dd x.
\]
\end{thm}
\begin{cor}\label{coro:Hausdorff dim of singular set}
    For any stationary weak solution, its singular set has zero $\mathcal H^{n-2}$ measure.
\end{cor}

\begin{rem}
Several comments on Theorem \ref{thm:weak-epsilon} are as follows.
\begin{enumerate}
\item Da Lio's three-dimensional argument uses the following Lorentz-space estimate for the Green potential on suitable spheres:
\[
 \norm{\nabla v}_{L^2(\partial B_\rho)}^2
 \leq C\norm{\nabla v}_{L^{2,\infty}(\partial B_\rho)}
       \norm{\nabla v}_{L^{2,1}(\partial B_\rho)}.
\]
The second factor is controlled by the embedding 
\[
W^{1,1}(\partial B_\rho)\hookrightarrow L^{2,1}(\partial B_\rho).
\]
In dimension $n$, the corresponding Sobolev exponent is $p=\frac{n-1}{n-2}$, whose conjugate is $p'=n-1$. These two exponents are both $2$ precisely when $n=3$. For $n>3$, the embedding gives $L^{p,1}$ control and the duality estimate would require $L^{n-1,\infty}$ control of the other gradient factor. This estimate does not hold in higher dimensions, so the arguments in \cite{DaLio} cannot be applied here.
\item If we impose the additional stability assumption,
\[
 \int_\Omega e^u\varphi^2\dd x\leq\int_\Omega\abs{\nabla\varphi}^2\dd x,
 \quad \forall\varphi\in C^\infty_0(\Omega),
\]
there is a more direct proof of this $\va$-regularity theorem, see the second author's work \cite{Wang12}.

\item Compared to the two dimensional result in Brezis-Merle \cite{BM91}, we need the Dirichlet energy condition in this $\va$-regularity theorem. We expect that this term cannot be removed. However, it can be relaxed, see Theorem \ref{thm:ba-smallmass}.

\item Our proof establishes existence of a dimensional threshold  $\varepsilon_n$ through compactness. It does not optimize its numerical value. Concerning  the singular set, we only have the  zero $\mathcal H^{n-2}$-measure conclusion and do not claim a sharper dimension bound, rectifiability of arbitrary singular sets, or a classification of all singular profiles at a singular point. 
\end{enumerate}
\end{rem}

\subsection{Difficulties and strategies}
A standard argument to obtain the $\ep$-regularity is by establishing an excess decay estimate. Here it reads as
\begin{align}\label{eqn:excess decay in introduction}
    \mathcal E_u\left(x,\frac{R}{4}\right)\leq\frac{1}{8}\mathcal E_u(x,R),
\end{align}
whenever the outer energy is sufficiently small. 

The main step to obtain this excess decay estimate is a harmonic approximation procedure for solutions with small energy. Fix $n\geq3$ and let $u_j\in H^1(B_2)$ be stationary weak solutions of $-\Delta u_j=e^{u_j}$, with $e^{u_j}\in L^1(B_2)$ and
\[
 E_j:=\int_{B_2}(\abs{\nabla u_j}^2+e^{u_j})\dd x\to0.
\]
Let $h_j$ be the harmonic function on $B_2$ with boundary trace $u_j$, and set $v_j=u_j-h_j$. Thus, $v_j\in H^1_0(B_2)$ solves $-\Delta v_j=e^{u_j}$. We normalize by setting 
\[
q_j=\frac{\nabla v_j}{\sqrt{E_j}},\qquad g_j=\frac{e^{u_j}}{E_j}.
\]
The energy orthogonality gives $\norm{q_j}_{L^2(B_2)}\leq1$. The key for this harmonic approximation procedure is to prove 
\begin{align}\label{key estimate on nonlinear effect}
    \norm{q_j}_{L^2(B_{\frac{1}{2}})}\to0,
\end{align}
 that is, a good control on the effect from the nonlinear term in the equation. Only with this estimate, we can say that the harmonic function $h_j$ is a good approximation to the original solution $u_j$.

For the proof of \eqref{key estimate on nonlinear effect}, as we will explain below, we need  control of $Dq_j$ in $L^1$, or equivalently, 
\begin{align}\label{L1 Hessian estimate}
    \norm{D^2v_j}_{L^1(B_1)}\leq C_nE_j.
\end{align}
However it is well known that the  Calderzon-Zygmund estimate fails at the end point,  so an $L^1$ bound on $g_j$ alone does not imply such an $L^1$ control. Instead, we use the estimate 
\[
\int_{B_2}e^{u_j}v_j\dd x\leq E_j
\]
and the harmonic bound $e^{h_j}\leq C_nE_j$ on $B_{\frac{3}{2}}$. These give a control
\[
 \int_{B_{\frac{3}{2}}}g_j\log^+g_j\dd x\leq C_n,\]
and in turn  \eqref{L1 Hessian estimate}. This is the main reason that we need the Dirichlet energy in the assumptions of $\ep$-regularity.

Combining \eqref{L1 Hessian estimate} with the Green estimate $\norm{\nabla v_j}_{L^1(B_2)}\leq C_nE_j$, we obtain
\[
q_j\to0\text{ in }W^{1,1}(B_1).
\]
We also have
\[
g_j-a_j\to0\text{ in }L^1(B_1),\quad \text{where}~~ a_j=\frac{e^{h_j(0)}}{E_j}.
\]
Subtracting the normalized harmonic stress and the constant $a_jI$ from the normalized stationary identity gives
\begin{align}\label{stationary field in introduction}
 \op{div}\left(q_j\otimes q_j-\frac{1}{2}\abs{q_j}^2I-R_j\right)=0,
 \quad \norm{R_j}_{L^1(B_1)}\to0.
\end{align}
Here $R_j$ consists of mixed gradient terms and the source difference $g_j-a_j$. The convergence of $ R_j $ follows from 
\[
\norm{\nabla h_j}_{L^\infty(B_1)}\leq C_n\sqrt{E_j}
\]
and the preceding estimates. See \eqref{eq:green-l1}, \eqref{eq:harmonic-bounds} and \eqref{eq:g-strong} in the proof of Lemma~\ref{lem:harmapprox}.

To prove \eqref{key estimate on nonlinear effect}, it suffices to   rule out the $L^2$-concentration for $q_j$.
The $W^{1,1}$ bounds alone do not imply this. We use a perturbed Sobolev estimate based on Ledoux \cite{Ledoux} and a maximal estimate to choose rescalings that preserve the vanishing errors. Persistent energy (in the sense of \eqref{eq:counter}, see Lemma~\ref{lem:critical}) would yield a non-zero positive matrix-valued measure $A$ on $\R^n$, the weak limit of the rescaled quadratic tensors, satisfying
\[
 \op{div}\left(A-\frac{1}{2}(\tr A)I\right)=0,
 \quad (\tr A)(B_R)\leq R^{n-2}\quad(R\geq1).
\]
The normalization in Lemma~\ref{lem:critical} and the monotonicity formula \eqref{eq:monotone} give $A(x)x=0$ outside $\overline B_1$. We choose the centers in this exterior region when applying Lemma~\ref{lem:gain}. As in the classical Federer dimension reduction principle, repeated changes of center and scale give $n-2$ constant directions with vanishing component energy. The planar Sobolev inequality on the remaining two-dimensional slices gives a contradiction, and hence $q_j\to0$ locally in $L^2$.

For a stationary solution $u$ on a ball $B_R(x)$, the harmonic approximation and $g_j-a_j\to0$ give the excess decay estimate \eqref{eqn:excess decay in introduction}  by contradiction and scaling, see  Lemma~\ref{lem:energy-improvement}.
Iteration bounds the Newtonian potential of $e^u$, and elliptic regularity then proves smoothness.

\subsection{Organization of this paper}
Section~\ref{sec:preliminaries} establishes the analytic and measure estimates used in the compactness argument. Section~\ref{sec:compactness} proves the compactness theorem for quadratic stresses (Theorem~\ref{thm:abstractcompact}) and derives a local energy estimate from it. Section~\ref{sec:liouville} applies them to the Liouville equation and proves Theorem~\ref{thm:weak-epsilon}. Section~\ref{sec:calibrations} tests the energy hypothesis and the contraction factor against explicit solutions. In Section~\ref{sec:blowup} we discuss a generalization of Theorem~\ref{thm:weak-epsilon} and the application of $\ep$-regularity results to the blow up analysis for the Liouville equation. 

\subsection{Notations and conventions}
We collect the conventions used throughout the paper.
\begin{itemize}
\item The dimension is a fixed integer $n\geq3$, and $s=n-2$. The notation $B_r(x)$ denotes the open Euclidean ball of radius $r$ centered at $x$, and $B_r=B_r(0)$. A superscript indicates the ambient dimension when needed.
\item Lebesgue measure is unnormalized. We write $\omega_n=\abs{B_1}$, use $\mathcal H^d$ for $d$-dimensional Hausdorff measure, and write $\mu\llcorner E$ for restriction of a measure to $E$. Compact inclusion is denoted by $\subset\subset$.
\item The derivative matrix is denoted by $Dq$, and matrix divergence is taken row by row: $(\op{div} R)_i=\sum_{k=1}^n\partial_k R_{ik}$. Tensor products satisfy $(q\otimes p)_{ik}=q_i p_k$, and $I$ is the identity matrix. Unwritten contractions in distributional stress identities use the Euclidean matrix inner product.
\item A positive semidefinite matrix-valued Radon measure $A$ is written $A=P\mu$, with $\mu=\tr A$, $P\geq0$, and $\tr P=1$ almost everywhere with respect to $\mu$. Assertions such as $Av=0$ and $A(x)x=0$ mean $Pv=0$ and $P(x)x=0$ almost everywhere with respect to $\mu$. If a matrix $A$ is semidefinite, we say it is nonnegative or $A\geq 0$.
\item Norms without a domain use the domain of the surrounding estimate. The constants $C_n$ may change from line to line and depend only on $n$. Additional dependence is indicated by a subscript. Regularity and pointwise estimates refer to the smooth representative of a weak solution.
\end{itemize}

\section{Preliminaries}\label{sec:preliminaries}
In this section, we establish the perturbed Sobolev estimate and the radial identity for stationary covariance measures that will be used in the compactness proof. General background on Radon measures and singular integrals is available in \cite{Mattila,Stein}.

\subsection{An improved Sobolev estimate with an \texorpdfstring{$L^2$}{L2} error}

The heat semigroup defines the negative Besov norm used in the perturbed estimate. Let $P_t=e^{t\Delta}$ be the Euclidean heat semigroup and set
\[
 \norm{a}_{\dot B^{-1}_{\infty,\infty}}
 =\sup_{t>0}\sqrt t\,\norm{P_ta}_{L^\infty}.
\]

\begin{lem}\label{lem:robust}
Suppose $f=a+b\in W^{1,1}(\R^n)\cap L^2(\R^n)$, where
\[
\norm{a}_{\dot B^{-1}_{\infty,\infty}}\leq L<+\infty,\quad b\in L^2(\R^n).
\]
Then
\begin{equation}\label{eq:robust}
 \norm{f}_2^2\leq C_n L\norm{Df}_1+C_n\norm{b}_2^2.
\end{equation}
The analogous estimate holds componentwise for vector-valued functions.
\end{lem}
\begin{proof}
We adapt the truncation argument of Ledoux \cite[Theorem~1]{Ledoux}. The case $L=0$ is immediate, so suppose $L>0$. Fix $c>144$. For $u>0$, put
\[
 t_u=\left(\frac{L}{u}\right)^2,
 \quad f_u=\operatorname{sgn}(f)\min\{(\abs{f}-u)_+,(c-1)u\}.
\]
Then
\[
 \abs{P_{t_u}f}\leq u+P_{t_u}\abs{b},
 \quad \abs{f-f_u}\leq u+\abs{f}\mathbf1_{\{\abs{f}>cu\}}.
\]
On $\{\abs{f}\geq6u\}$ we have $\abs{f_u}\geq5u$. Consequently, that set (that is, $\{\abs{f}\geq6u\}$) is contained in the union of
\[
 \{\abs{f_u-P_{t_u}f_u}\geq u\},\quad
 \{P_{t_u}\abs{b}\geq u\},\quad
 \{P_{t_u}(\abs{f}\mathbf1_{\{\abs{f}>cu\}})\geq u\}.
\]
Integrate their measures with respect to $\dd (u^2)$. The left side is $\frac{\norm{f}_2^2}{36}$. Denote the three sets above, in the displayed order, by $S_1(u),S_2(u),S_3(u)$, and put
\begin{equation}\label{eq:robust-contributions}
\begin{aligned}
 I_k&:=\int_0^{+\infty}\abs{S_k(u)}\dd(u^2),\quad k=1,2,3,\\
 \frac{\norm{f}_2^2}{36}
 &=\int_0^{+\infty}\abs{\{\abs{f}\geq6u\}}\dd(u^2)
 \leq I_1+I_2+I_3.
\end{aligned}
\end{equation} The heat pseudo-Poincar\'e inequality \cite[equation~(4)]{Ledoux}
\[
 \norm{g-P_tg}_1\leq C_n\sqrt t\,\norm{Dg}_1
\]
and the Sobolev chain rule show that the first contribution  $I_1$ in \eqref{eq:robust-contributions}  is at most
\[
 2C_nL\int_0^{+\infty}\frac{\dd u}{u}
 \int_{\{u<\abs{f}<cu\}}\abs{Df}\dd x
 =2C_nL\log(c)\norm{Df}_1.
\]
The second contribution  $I_2$ in \eqref{eq:robust-contributions}  is at most $C_n\norm{b}_2^2$, by the $L^2$ boundedness of the heat maximal operator. For the third,  which is $I_3$ in \eqref{eq:robust-contributions},  positivity and preservation of the integral by $P_t$, followed by Fubini theorem, yield
\[
 \int_0^{+\infty}\frac{1}{u}\int_{\{\abs{f}>cu\}}\abs{f}\dd x\dd (u^2)
 =\frac{2}{c}\norm{f}_2^2.
\]
Absorbing this last term  into the left-hand side of \eqref{eq:robust-contributions}, using $\frac{1}{36}-\frac{2}{c}>0$,  proves \eqref{eq:robust}. The vector-valued version follows by summing the scalar estimates for its components.
\end{proof}

\begin{rem}
We also use the elementary implication
\begin{equation}\label{eq:morreybesov}
 \int_{B_r(x)}\abs{a}^2\dd x\leq M r^{n-2}\quad\hbox{for every }x,r
 \quad\Rightarrow\quad
 \norm{a}_{\dot B^{-1}_{\infty,\infty}}\leq C_n\sqrt M.
\end{equation}
Indeed, Cauchy--Schwarz inequality bounds $\int_{B_r(x)}\abs{a}$ by $C_n\sqrt M r^{n-1}$. Summing the heat kernel over dyadic annuli yields \eqref{eq:morreybesov}.
\end{rem}

\subsection{Radial monotonicity for covariance measures}

In this subsection, following Moser \cite{Moser}, we establish a monotonicity formula for  the   density measure of a stationary field. The proof uses the positive semidefiniteness combined with the radial stationary identity. 

\begin{lem}\label{lem:monotone}
Let $A$ be a positive semidefinite matrix-valued Radon measure on $\R^n$. Let $\mu=\tr A$, and suppose $\op{div}(A-\frac{1}{2}\mu I)=0$. At continuity radii $0<r<R$,
\begin{equation}\label{eq:monotone}
 R^{-s}\mu(B_R)-r^{-s}\mu(B_r)
 =2\int_{B_R\backslash B_r}\abs{x}^{-s}
 \left(\frac{x}{\abs{x}}\right)^T A(\dd x)\left(\frac{x}{\abs{x}}\right)\geq0.
\end{equation}
In particular $\mu(\{0\})=0$ and $r^{-s}\mu(B_r)$ has a non-decreasing representative.
\end{lem}
\begin{proof}
First mollify $A$ and $\mu$ with the same non-negative compact mollifier. This preserves positive semidefiniteness, the trace relation, and the linear divergence equation. For the resulting smooth fields, the divergence theorem applied to $(A-\frac{1}{2}\mu I)x$ yields
\[
 -\frac{s}{2}\int_{B_t}\mu\dd x
 =t\int_{\partial B_t}\left(e_r^TAe_r-\frac{1}{2}\mu\right)\dd \sigma.
\]
Differentiating $t^{-s}\int_{B_t}\mu$ and integrating in $t$ proves \eqref{eq:monotone} for the mollified fields. Passing to the limit at continuity radii proves the formula for the measures. All weights are bounded and continuous on the annulus, and its boundary has zero $\mu$ measure. Monotonicity bounds $\mu(B_r)$ by a constant times $r^s$ as $r\downarrow0$, so there is no atom at the origin.
\end{proof}

\section{Compactness for quadratic stresses}\label{sec:compactness}

The following compactness theorem is a key ingredient in the proof of Theorem~\ref{thm:weak-epsilon}, as it yields the harmonic approximation needed for the excess decay.
For $q\in\R^n$, write
\[
 \cT(q)=q\otimes q-\frac{1}{2}\abs{q}^2I.
\]
\begin{thm}\label{thm:main}\label{thm:abstractcompact}
Let $n\geq3$. Suppose
\[
 q_j\in W^{1,1}(B_2;\R^n)\cap L^2(B_2;\R^n),
 \quad R_j\in L^1(B_2;\R^{n\times n}),
\]
satisfy
\[
\begin{aligned}
 &\norm{q_j}_{L^1(B_2)}+\norm{Dq_j}_{L^1(B_2)}+\norm{R_j}_{L^1(B_2)}\to0,\\
 &\sup_j\norm{q_j}_{L^2(B_2)}<+\infty,
 \quad\op{div}\left(\cT(q_j)-R_j\right)=0\quad\hbox{in }\mathcal D'(B_2).
\end{aligned}
\]
Then $q_j\to0$ in $L^2(B_1)$.
\end{thm}

\begin{rem}
We make several remarks on the above theorem.
\begin{enumerate} 
\item No curl condition is imposed. In the Liouville application, $q_j$ is a scalar gradient, but the proof uses only the stated hypotheses. The tensor error $R_j$ may be any $L^1$ field whose norm tends to zero.
\item The limiting identity is the quadratic case of the stationary matrix-measure theory studied by Moser \cite{Moser}. The measure identities needed here are proved in Section~\ref{sec:preliminaries}. In dimension three, the exact vector stress also occurs in the weak Beltrami-flow formulation of Chae and Wolf \cite{ChaeWolf}, whose radial identity gives local critical Morrey control. Limiting measures and their invariances are also used in Lin's blow-up analysis of stationary harmonic maps \cite{Lin}.
\item This compactness theorem implies  a local energy estimate, see  Corollary~\ref{cor:coercivity}  below. In particular, the uniform outer $L^2$ bound can be omitted from the sequential conclusion, provided each $q_j$ belongs to $L^2(B_2)$.
\item The proof of Theorem~\ref{thm:main} rules out persistent energy by induction on the number of constant directions in which the component energy vanishes. This is a kind of dimension reduction procedure similar to the classical Federer dimension reduction principle. The induction ends with a two-dimensional slicing argument.
\end{enumerate}
\end{rem}

\subsection{Counterexample classes and critical normalization}

Recall that $s=n-2$. For a linear subspace $V\subset\R^n$, call a sequence a counterexample in $\mathcal C(V)$ if it satisfies the hypotheses of Theorem~\ref{thm:main} and, for fixed $c>0$ and $\Lambda<+\infty$,
\begin{equation}\label{eq:counter}
 \int_{B_1}\abs{q_j}^2\dd x\geq c,
 \quad \int_{B_2}\abs{q_j}^2\dd x\leq\Lambda,
 \quad\int_{B_2}\abs{\operatorname{proj}_Vq_j}^2\dd x\to0.
\end{equation}
The constants $c,\Lambda$ may depend on the sequence. Theorem \ref{thm:abstractcompact} is equivalent to the claim that $\mathcal C(\{0\})$ is empty.
We reduce this claim to the one $\mathcal{C}(V)=\emptyset$ for some subspace $V$ of dimension $n-2$ in this and the next subsections,  by using a dimension reduction argument.

\begin{lem}\label{lem:critical}
If $\mathcal C(V)$ is non-empty, there are fields $Q_j,S_j$ on balls exhausting $\R^n$ with
\[
\begin{aligned}
 &\op{div}(\cT(Q_j)-S_j)=0,\\
 &\norm{Q_j}_{L^1(K)}+\norm{DQ_j}_{L^1(K)}+\norm{S_j}_{L^1(K)}\to0,\\
 &\int_K\abs{\operatorname{proj}_VQ_j}^2\dd x\to0
 \quad(K\subset\subset\R^n),
\end{aligned}
\]
and a non-zero pair of Radon measures $A,\mu$ on $\R^n$ such that
\begin{align}
 Q_j\otimes Q_j\dd x&\stackrel{*}{\rightharpoonup}A,
 &\abs{Q_j}^2\dd x&\stackrel{*}{\rightharpoonup}\mu,\label{eq:covconv}\\
 A&\geq0,
 &\tr A&=\mu,\quad Av=0\quad(v\in V),\label{eq:covprop}\\
 \op{div}\left(A-\frac{1}{2}\mu I\right)&=0,
 &\mu(B_R)&\leq R^s\ (R\geq1),\quad \mu(\overline B_1)\geq1.
 \label{eq:globalgrowth}
\end{align}
Here $A\geq0$ means positive semidefinite as a matrix-valued measure.

The normalization also gives
\[
\begin{aligned}
 &\mu(\overline B_1)=1,\quad \mu(B_R)=R^s\quad(R>1),\\
 &A(x)x=0\quad\text{for }\mu\text{-almost every }x\in\R^n\backslash\overline B_1.
\end{aligned}
\]
\end{lem}
\begin{proof}
Take a sequence in $\mathcal C(V)$ and write
\[
 \mu_j=\abs{q_j}^2\dd x,\quad
 \beta_j=(\abs{Dq_j}+\abs{R_j}+\abs{\operatorname{proj}_Vq_j}^2)\dd x,
\]
\[
 \eta_j=\norm{q_j}_{L^1(B_2)}+\beta_j(B_2)+\frac{1}{j},
 \quad\kappa_j=\eta_j^{\frac{1}{4}},\quad m_j=\eta_j^{-\frac{1}{4}},
 \quad\rho_0=\frac{1}{8}.
\]
Define the set of good centers
\[
 G_j=\{x\in B_{\frac{3}{2}}:\beta_j(B_r(x))\leq\kappa_j\mu_j(B_r(x))
 \text{ for all }0<r\leq\rho_0\}.
\]
The centered maximal-ratio weak estimate yields
\begin{equation}\label{eq:badmass}
 \mu_j\left(B_{\frac{3}{2}}\backslash G_j\right)
 \leq C_n\kappa_j^{-1}\beta_j(B_2)
 \leq C_n\eta_j^{\frac{3}{4}}.
\end{equation}
This estimate uses no doubling hypothesis on $\mu_j$: apply the Besicovitch covering theorem to balls centered in the bad set on which $\beta_j(B)>\kappa_j\mu_j(B)$, select a covering of bounded multiplicity, and sum $\mu_j(B)<\kappa_j^{-1}\beta_j(B)$. For the centered covering theorem used here, see \cite[Theorem~3.1]{GrafakosKinnunen}. All these balls lie in $B_2$. The measures here have $L^1$ densities, so continuity in the radius also makes $G_j$ a Borel set by restricting first to rational radii.

{\bf Claim.}  For every sufficiently large $j$, there are $x_j\in G_j$ and a radius $\widehat{r}_j$ such that
\begin{align}\label{eqn:existence of bad radius}
    \widehat{r}_j^{-s}\mu_j(B_{\widehat{r}_j}(x_j))> m_j. 
\end{align}

Assume to the contraty that, along a subsequence, the bound
\begin{equation}\label{eq:assumelow}
 r^{-s}\mu_j(B_r(x))\leq m_j
 \quad\text{holds for every }x\in G_j\text{ and }0<r\leq\rho_0.
\end{equation}

Choose $\psi\in C_c^\infty(B_{\frac{3}{2}})$, $0\leq\psi\leq1$, equal to one on $B_1$, and extend $\psi q_j$ by zero. Split
\[
 \psi q_j=a_j+b_j,
 \quad a_j=\psi q_j\mathbf1_{G_j},\quad
 b_j=\psi q_j\mathbf1_{G_j^c}.
\]
For a ball $B_r(z)$ on which $\int_{B_r(z)}\abs{a_j}^2>0$, choose $x\in G_j\cap B_r(z)$. Balls with zero integral require no estimate. For $2r\leq\rho_0$, \eqref{eq:assumelow} yields
\[
 \int_{B_r(z)}\abs{a_j}^2\dd x\leq\mu_j(B_{2r}(x))\leq m_j(2r)^s.
\]
For larger radii the uniform energy bound $\Lambda$ yields the same estimate with $C_nm_j$, for all large $j$. Thus \eqref{eq:morreybesov} yields $\norm{a_j}_{\dot B^{-1}_{\infty,\infty}}\leq C_n\sqrt{m_j}$. Also
\[
 \norm{b_j}_2^2\leq C_n\eta_j^{\frac{3}{4}},
 \quad\norm{D(\psi q_j)}_1\leq C_n\eta_j.
\]
Lemma~\ref{lem:robust} would imply
\[
 c\leq\norm{\psi q_j}_2^2
 \leq C_n\sqrt{m_j}\eta_j+C_n\eta_j^{\frac{3}{4}}\to0.
\]
This is a contradiction, and the claim follows.

The density at $\rho_0$ is at most $\rho_0^{-s}\Lambda<m_j$ for all large $j$. By the above claim, we choose the outermost radius $r_j\in(\widehat{r}_j,\rho_0)$ of the level $m_j$. Continuity yields
\begin{equation}\label{eq:outermost}
 \mu_j(B_{r_j}(x_j))=m_jr_j^s,
 \quad \mu_j(B_t(x_j))\leq m_jt^s\quad(r_j\leq t\leq\rho_0).
\end{equation}
In particular $r_j^s\leq\frac{\Lambda}{m_j}\to0$. On $B_{\frac{\rho_0}{r_j}}$ set
\[
 Q_j(z)=\frac{r_j}{\sqrt{m_j}}q_j(x_j+r_jz),
 \quad S_j(z)=\frac{r_j^2}{m_j}R_j(x_j+r_jz).
\]
For every fixed $R\geq1$, goodness and \eqref{eq:outermost} imply, for large $j$,
\[
\begin{aligned}
 &\int_{B_R}\abs{Q_j}^2\dd x\leq R^s,
 \quad\int_{B_1}\abs{Q_j}^2\dd x=1,\\
 &\int_{B_R}\abs{DQ_j}\dd x\leq\kappa_j\sqrt{m_j}R^s\to0,\\
 &\int_{B_R}\abs{S_j}\dd x+
 \int_{B_R}\abs{\operatorname{proj}_VQ_j}^2\dd x
 \leq2\kappa_jR^s\to0.
\end{aligned}
\]
Local compactness of $W^{1,1}$ in $L^1$ provides a constant vector limit $Q_j\to a$ in $L^1_{\rm loc}$ after subsequence. The energy bound for every $R\geq1$ yields $\abs{a}^2\abs{B_R}\leq R^s$. Sending $R\to+\infty$ forces $a=0$.

Extracting the covariance measures now yields \eqref{eq:covconv}--\eqref{eq:globalgrowth}. Positive semidefiniteness and vanishing $V$-components pass to the limit. Non-zero mass follows by testing cutoffs equal to one on $\overline B_1$, then shrinking their supports to that ball. The stress equation passes to the limit because $S_j\to0$ locally in $L^1$.

For a continuity radius $R>1$, let $r\downarrow1$ through continuity radii in \eqref{eq:monotone}. Together with \eqref{eq:globalgrowth}, this gives
\[
 1\leq\mu(\overline B_1)\leq R^{-s}\mu(B_R)\leq1.
\]
Thus $\mu(\overline B_1)=1$ and $\mu(B_R)=R^s$ at every continuity radius $R>1$. Continuity from below extends the latter identity to all $R>1$. Both endpoint densities in \eqref{eq:monotone} therefore equal one on every annulus with continuity radii $1<r<R$. Its non-negative right-hand side vanishes there. Positive semidefiniteness gives $A(x)x=0$ for $\mu$-almost every $x$ with $\abs{x}>1$, proving the additional assertions.
\end{proof}

\subsection{Choice of a zero direction}

Now we use the radial kernel outside $\overline B_1$ from Lemma~\ref{lem:critical} and the blow-up method to gain a new constant zero direction.

\begin{lem}\label{lem:gain}
If $\mathcal C(V)$ is non-empty and $\dim V<s$, then $\mathcal C(V\oplus\R e)$ is non-empty for some unit vector $e\perp V$.
\end{lem}
\begin{proof}

Take the pair $(A,\mu)$ from Lemma~\ref{lem:critical}, and put $k=\dim V<s$. For every $v\in V$, stationarity and $Av=0$ imply $\partial_v\mu=0$. Thus $\mu$ is translation invariant along $V$. We claim that
\[
 \mu\left(\{x:\abs{x}>1\}\backslash V\right)>0.
\]
If $V=\{0\}$, this follows from $\mu(B_R\backslash\overline B_1)=R^s-1>0$ for $R>1$. Otherwise, if the claim failed, any bounded Borel subset of $\R^n\backslash V$ could be translated along $V$ into $\{x:\abs{x}>1\}\backslash V$ and would have zero measure. Hence $\mu$ would be supported on $V$ and, by translation invariance, would equal a constant multiple of $\mathcal H^k\llcorner V$. This contradicts $\mu(B_R)=R^s$ for every $R>1$, since $k<s$.

Choose a $\mu$-typical point $x_0\notin\overline B_1\cup V$ and radii $r_\ell\downarrow0$ with $2r_\ell<\abs{x_0}-1$, such that both $r_\ell$ and $2r_\ell$ are continuity radii and

\begin{equation}\label{eq:doubling}
 0<\mu(B_{r_\ell}(x_0)),\quad
 \mu(B_{2r_\ell}(x_0))\leq C_n\mu(B_{r_\ell}(x_0)).
\end{equation}
For completeness, such subsequences exist at almost every point of every locally finite Radon measure: if $\mu(B_{4r}(x))>4^{n+1}\mu(B_r(x))$ at all sufficiently small scales, iteration yields $\mu(B_r(x))\leq C_xr^{n+1}$. The set of such points is $\mu$-null: partition it into bounded subsets on which the constants and admissible radii are uniform, cover each by $O(r^{-n})$ balls centered in that subset, and let the resulting $O(r)$ mass bound tend to zero. After choosing factor-four doubling radii, perturb each within $[r,2r]$ to avoid the countably many non-continuity spheres. This preserves a fixed factor-two doubling constant.

Set
\[
 a=\abs{\operatorname{proj}_{V^\perp}x_0}>0,
 \quad e=\frac{\operatorname{proj}_{V^\perp}x_0}{a}.
\]
By Lemma~\ref{lem:critical} and $2r_\ell<\abs{x_0}-1$, we have $A(y)y=0$ on $B_{2r_\ell}(x_0)$. Since $A$ also annihilates $V$, positive semidefiniteness yields, on this ball,
\begin{equation}\label{eq:newdir}
 e^TAe\leq\frac{\abs{y-x_0}^2}{a^2}\dd \mu(y)
 \leq\frac{4r_\ell^2}{a^2}\dd \mu(y).
\end{equation}

Let $(Q_j,S_j)$ realize $(A,\mu)$ as in Lemma~\ref{lem:critical}, and write $M_\ell=\mu(B_{r_\ell}(x_0))$. For each fixed $\ell$, choose $j=j(\ell)$ sufficiently large and define on $B_2$
\[
 \widehat q_\ell(z)=\frac{r_\ell^{\frac{n}{2}}}{\sqrt{M_\ell}}
 Q_{j(\ell)}(x_0+r_\ell z),\quad
 \widehat R_\ell(z)=\frac{r_\ell^n}{M_\ell}
 S_{j(\ell)}(x_0+r_\ell z).
\]
The transformed stress equation is exact. By continuity radii, choose the indices so that the energy on $B_1$ lies between $\frac{1}{2}$ and $2$, and the energy on $B_2$ is at most $2C_n$. Choose them still later, if necessary, so that the $L^1$ norms of $\widehat q_\ell,D\widehat q_\ell,\widehat R_\ell$ and the squared $L^2$ norm of the old $V$-components are below $\frac{1}{\ell}$. This is possible because all scaling factors are finite and fixed before the index is selected. Finally, \eqref{eq:doubling}--\eqref{eq:newdir} and covariance convergence yield
\[
 \int_{B_2}\abs{\widehat q_\ell\cdot e}^2\dd x
 \leq\frac{4C_nr_\ell^2}{a^2}+\frac{1}{\ell}\to0.
\]
The resulting sequence belongs to $\mathcal C(V\oplus\R e)$.
\end{proof}

\subsection{The terminal two-dimensional slicing argument}

The induction procedure in the previous lemma terminates when only two transverse field components remain. This terminal case can be handled by the planar Sobolev inequality, which will lead to a contradiction with the persistence of the energy.

\begin{lem}\label{lem:constancy}
On a ball in $\R^d$, suppose $F_j$ is bounded in $L^1_{\rm loc}$, $F_j\to\theta$ in distributions for a constant $\theta$, and
\[
 \nabla F_j=\op{div} E_j+b_j,
 \quad E_j\to0\text{ in }L^1_{\rm loc},\quad
 b_j\to0\text{ in }L^1_{\rm loc}.
\]
Then $F_j\to\theta$ in measure on smaller balls.
\end{lem}
\begin{proof}
First mollify in the $d$ variables at an index-dependent scale, so that the local $L^1$ difference from $F_j$ tends to zero. This preserves the identity on slightly smaller balls and does not increase the source $L^1$ norms. Multiply the mollified $E_j,b_j$ by a fixed cutoff equal to one on an intermediate ball and solve on $\R^d$
\[
 \Delta H_j=\sum_{i,k}\partial_i\partial_k(E_j)_{ik}
 +\sum_i\partial_i(b_j)_i.
\]
For $d\geq2$, each operator $\partial_i\partial_k\Delta^{-1}$ is bounded on $L^2$ by its Fourier multiplier $\frac{\xi_i\xi_k}{\abs{\xi}^2}$. Away from the diagonal, its kernel and first derivatives are bounded by $C_d\abs{x-y}^{-d}$ and $C_d\abs{x-y}^{-d-1}$, respectively. It is therefore a Calder\'on--Zygmund operator in the sense of \cite[Definitions~2.1 and~2.4]{TaoNotes}, and the weak type $(1,1)$ estimate in \cite[Corollary~2.9]{TaoNotes} makes the zero-order part tend to zero in measure. Its distributional convergence follows separately by duality: the adjoint applied to a smooth test function is bounded on the common compact support of the sources.

The first-order part tends to zero locally in $L^1$. Indeed, if $\Gamma_d$ is the fundamental solution of $\Delta$ and $S$ is the common compact source support, then for each compact set $K$,
\[
 \int_K\int_S\abs{\nabla\Gamma_d(x-y)}\,\abs{b_j(y)}\dd y\dd x
 \leq C_{K,S,d}\norm{b_j}_1\to0,
\]
by Fubini and local integrability of $\abs{x}^{1-d}$. Here $E_j,b_j$ denote the mollified, cutoff sources. In dimension $d=1$, the zero-order part is an $L^1$ error and the first-order part is a primitive bounded by $\norm{b_j}_1$.

Thus $H_j\to0$ both in measure locally and in distributions. The harmonic difference between the mollified $F_j$ and $H_j$ converges to $\theta$ in distributions, and hence smoothly on smaller balls: a fixed radial mollifier reproduces harmonic functions there, and derivatives can be transferred to that mollifier. Returning to $F_j$ using the small mollification error, we complete the proof.
\end{proof}

\begin{rem}\label{rem:strong-constancy}
For non-negative $F_j$, the conclusion of Lemma~\ref{lem:constancy} improves to strong convergence in $L^1_{\rm loc}$, as in Allard's strong constancy lemma \cite{Allard86}. Indeed, $\theta\geq0$, and for every non-negative smooth cutoff $\chi$ supported in the ball,
\[
 \int\chi\abs{F_j-\theta}\dd x
 =\int\chi(F_j-\theta)\dd x
 +2\int\chi(\theta-F_j)_+\dd x\to0.
\]
The first term tends to zero by distributional convergence. The second does so by convergence in measure and the bound $0\leq(\theta-F_j)_+\leq\theta$. This stronger conclusion applies to the non-negative slice energies in Lemma~\ref{lem:terminal}.
\end{rem}

\begin{lem}\label{lem:terminal}
If $\dim V=n-2$, then $\mathcal C(V)$ is empty.
\end{lem}
\begin{proof}
Rotate coordinates so that $\R^n=\R^{n-2}_z\times\R^2_t$ and $V$ is the $z$-space. The cylinder $B^{n-2}_{\frac{5}{4}}\times B^2_{\frac{5}{4}}$ is compactly contained in $B_2$ and contains $B_1$. Choose $\eta\in C_c^\infty(B^2_{\frac{5}{4}})$ with $0\leq\eta\leq1$, equal to one on $B^2_1$, and set
\[
 F_j(z)=\int\eta(t)^2\abs{q_j(z,t)}^2\dd t.
\]
These non-negative functions are locally bounded in $L^1_z$. Passing to a subsequence produces a weak measure limit. For tangent indices $i,k\leq n-2$, integrate the $i$th stress equation in the $t$ variables to obtain
\[
 \partial_iF_j=\sum_k\partial_k(E_j)_{ik}+(b_j)_i,
 \quad (E_j)_{ik}=2\int\eta^2\left((q_j)_i(q_j)_k-(R_j)_{ik}\right)\dd t.
\]
The components of the remaining vector are explicitly
\[
 (b_j)_i=-2\sum_{a=1}^2\int
 \partial_{t_a}(\eta^2)\left((q_j)_i(q_j)_{t_a}-(R_j)_{i,t_a}\right)\dd t.
\]
Vanishing $V$-component energy and Cauchy--Schwarz yield
\[
 \norm{E_j}_1\to0,
 \quad\norm{b_j}_1\leq C\norm{\operatorname{proj}_Vq_j}_2\norm{q_j}_2+C\norm{R_j}_1\to0.
\]
The limiting measure $F$ has $\nabla F=0$, hence is $\theta\dd z$ for a constant $\theta\geq0$. This constant is positive: $\int_{B_1^{n-2}}F_j\dd z\geq\int_{B_1}\abs{q_j}^2\dd x\geq c$, and a cutoff equal to one on $\overline B_1^{n-2}$ transfers this positive mass to the limit. Lemma~\ref{lem:constancy} implies $F_j\to\theta$ in measure on interior $z$-balls.

On the other hand, the two-dimensional Sobolev inequality, applied to $\eta q_j(z,\cdot)$ for almost every $z$, yields
\[
 \sqrt{F_j(z)}\leq C\int\left(\eta\abs{D_tq_j}+\abs{D_t\eta}\,\abs{q_j}\right)\dd t.
\]
Its right-hand side tends to zero in $L^1_z$. Therefore $F_j\to0$ in measure, contradicting $\theta>0$.
\end{proof}

\begin{proof}[Proof of Theorem~\ref{thm:main}]
If the conclusion failed, a subsequence would belong to $\mathcal C(\{0\})$. Repeatedly applying Lemma~\ref{lem:gain} produces a non-empty $\mathcal C(V)$ with $\dim V=n-2$. This contradicts Lemma~\ref{lem:terminal}.
\end{proof}

\subsection{A local energy estimate}
The qualitative compactness in Theorem \ref{thm:abstractcompact} can be turned into a coercive estimate, by using
amplitude normalization and an iteration.
\begin{cor}\label{cor:coercivity}
There is $C_n<+\infty$ such that every
\[
 q\in W^{1,1}(B_2;\R^n)\cap L^2(B_2;\R^n),\quad
 R\in L^1(B_2;\R^{n\times n}),\quad \op{div}(\cT(q)-R)=0,
\]
satisfies
\begin{equation}\label{eq:coercivity}
 \int_{B_1}\abs{q}^2\dd x
 \leq C_n\left[
 \left(\norm{q}_{L^1(B_2)}+\norm{Dq}_{L^1(B_2)}\right)^2
 +\norm{R}_{L^1(B_2)}\right].
\end{equation}
Consequently the uniform outer $L^2$ bound in the sequence statement can be omitted, provided every field individually belongs to $L^2(B_2)$.
\end{cor}
\begin{proof}
First, for every $\gamma>0$, Theorem~\ref{thm:main} and amplitude homogeneity imply
\begin{equation}\label{eq:gamma-interpolation}
 \int_{B_1}\abs{q}^2\dd x
 \leq\gamma\int_{B_2}\abs{q}^2\dd x
 +C_{n,\gamma}\left[
 \left(\norm{q}_1+\norm{Dq}_1\right)^2+\norm{R}_1\right],
\end{equation}
where all norms on the right are on $B_2$. Indeed, if this failed for some fixed $0<\gamma<1$, choose successive counterexamples with coefficient $j$ in place of $C_{n,\gamma}$ and divide $q$ by its outer $L^2$ norm and $R$ by the square of that norm. The normalized outer energy is one, the rightmost small quantities tend to zero, and the inner energy exceeds $\gamma$, contrary to Theorem~\ref{thm:main}. For $\gamma\geq1$ the assertion is immediate.

Write $F(r)=\int_{B_r}\abs{q}^2$, $B=\norm{q}_{L^1(B_2)}+\norm{Dq}_{L^1(B_2)}$, and $\mathcal R=\norm{R}_{L^1(B_2)}$. Given $1\leq r<s\leq2$, put $\delta=\frac{s-r}{4}$. Cover $B_r$ by finitely many balls $B_\delta(x_i)$ centered in $B_r$ whose doubled balls lie in $B_s$ and have overlap at most a dimensional constant $N_n$. Applying \eqref{eq:gamma-interpolation} after the change of variables $x=x_i+\delta y$ yields
\[
 \int_{B_\delta(x_i)}\abs{q}^2
 \leq\gamma\int_{B_{2\delta}(x_i)}\abs{q}^2
 +C_{n,\gamma}\left[\delta^{-n}(Q_i+\delta D_i)^2+\mathcal R_i\right],
\]
where $Q_i,D_i,\mathcal R_i$ are the integrals of $\abs{q},\abs{Dq},\abs{R}$ on $B_{2\delta}(x_i)$. Bounded overlap yields
\[
 \sum_i(Q_i+\delta D_i)^2
 \leq\left(\sum_i(Q_i+\delta D_i)\right)^2\leq N_n^2B^2,
 \quad\sum_i\mathcal R_i\leq N_n\mathcal R.
\]
Choose the fixed value $\gamma=(2N_n)^{-1}$ and sum to obtain
\begin{equation}\label{eq:hole-filling}
 F(r)\leq\frac{1}{2} F(s)+C_n(s-r)^{-n}B^2+C_n\mathcal R.
\end{equation}
There is no factor equal to the number of covering balls in this estimate.

Let $\tau=2^{-\frac{1}{2n}}$ and $r_k=2-\tau^k$, so $r_0=1$ and $r_k\uparrow2$. Iterating \eqref{eq:hole-filling} yields
\[
 F(1)\leq2^{-K}F(r_K)
 +C_n(1-\tau)^{-n}B^2\sum_{k=0}^{K-1}(2^{-1}\tau^{-n})^k
 +C_n\mathcal R\sum_{k=0}^{K-1}2^{-k}.
\]
The first term tends to zero because this individual field has finite $F(2)$. The geometric ratio in the second sum is $2^{-\frac{1}{2}}<1$. Sending $K\to+\infty$ proves \eqref{eq:coercivity}.
\end{proof}

\section{Proof of Theorem \ref{thm:weak-epsilon}}\label{sec:liouville}

In this section, we prove Theorem \ref{thm:weak-epsilon}.
The idea is to apply the quadratic-stress estimate in the  previous section to the normalized Green potential and then improve and iterate the excess decay.

\subsection{Normalized harmonic approximation}

We first show that the Dirichlet Green part carries a vanishing fraction of the total energy as that energy tends to zero.

\begin{lem}\label{lem:harmapprox}
Let $u_j\in H^1(B_2)$ be stationary weak solutions of $-\Delta u_j=e^{u_j}$, with $e^{u_j}\in L^1(B_2)$, and set
\[
 E_j=\int_{B_2}(\abs{\nabla u_j}^2+e^{u_j})\dd x\to0.
\]
On $B_2$ write $u_j=h_j+v_j$, where $h_j$ is harmonic with boundary trace $u_j$ and $v_j\in H^1_0(B_2)$. Then
\[
 E_j^{-1}\int_{B_{\frac{1}{2}}}\abs{\nabla v_j}^2\dd x\to0.
\]
\end{lem}

\begin{proof}
Write $f_j=e^{u_j}$ and $M_j=\int_{B_2}f_j\dd x$. The weak equation yields $-\Delta v_j=f_j$. Positivity of the Dirichlet Green function, or the weak maximum principle, yields $v_j\geq0$. Energy orthogonality yields
\[
 \int_{B_2}\abs{\nabla v_j}^2+\abs{\nabla h_j}^2\dd x
 =\int_{B_2}\abs{\nabla u_j}^2\dd x\leq E_j.
\]
Testing with $\min(v_j,L)$ and letting $L\to+\infty$ yields the exact identity
\begin{equation}\label{eq:fv-energy}
 \int_{B_2}f_jv_j\dd x=\int_{B_2}\abs{\nabla v_j}^2\dd x\leq E_j.
\end{equation}
The bounded truncations are legitimate $H^1_0$ tests. Indeed, there are uniformly bounded smooth compactly supported functions converging to $\min(v_j,L)$ in $H^1_0$ and almost everywhere. Dominated convergence against $f_j\in L^1$ identifies the displayed integral with the $H^{-1}$ pairing furnished by the weak equation.

The Green bounds $G(x,y)\leq C_n\abs{x-y}^{2-n}$ and $\abs{\nabla_xG(x,y)}\leq C_n\abs{x-y}^{1-n}$ yield
\begin{equation}\label{eq:green-l1}
 \norm{v_j}_{L^1(B_2)}+\norm{\nabla v_j}_{L^1(B_2)}
 \leq C_n\norm{f_j}_{L^1(B_2)}\leq C_nE_j.
\end{equation}
Since $e^{h_j}$ is subharmonic and $h_j\leq u_j$, interior mean-value and harmonic estimates yield
\begin{equation}\label{eq:harmonic-bounds}
 e^{h_j}\leq C_nE_j,\quad
 \abs{\nabla h_j}\leq C_n\sqrt{E_j} \quad \text{ on} ~ B_{\frac{3}{2}}.
\end{equation}

Let $g_j=\frac{f_j}{E_j}$. Then $\int_{B_2} g_j\leq1$ and, on $B_{\frac{3}{2}}$,
\[
 \log^+g_j\leq v_j+C_n.
\]
Together with \eqref{eq:fv-energy}, this yields
\begin{equation}\label{eq:normalized-entropy}
 \int_{B_{\frac{3}{2}}}g_j\log^+g_j\dd x\leq C_n.
\end{equation}
In particular $(g_j)$ is uniformly integrable there. Define the spatial constant
\[
 a_j=\frac{e^{h_j(0)}}{E_j}.
\]
Centered subharmonicity on the full ball $B_2$, followed by $e^{h_j}\leq f_j$, yields
\begin{equation}\label{eq:constant-mass-bound}
 0<a_j\leq\frac{M_j}{\abs{B_2}E_j}\leq\frac{1}{\abs{B_2}}.
\end{equation}
The convergence to this moving constant can be quantified directly. Put $\delta_j=h_j-h_j(0)$ and use the exact identity
\[
 g_j-a_j=g_j(1-e^{-\delta_j})+a_j(e^{v_j}-1).
\]
The first term has $L^1(B_{\frac{3}{2}})$ norm at most $C_n\sqrt{E_j}$, since $\norm{\delta_j}_\infty\leq C_n\sqrt{E_j}$ and $\int_{B_2} g_j\leq1$. For $L\geq1$, split the non-negative second term into $\{v_j\leq L\}$ and $\{v_j>L\}$. Equations~\eqref{eq:fv-energy}--\eqref{eq:green-l1} yield
\[
\begin{aligned}
 \int_{\{v_j\leq L\}\cap B_{\frac{3}{2}}}a_j(e^{v_j}-1)\dd x
 &\leq C_n E_j e^L,\\
 \int_{\{v_j>L\}\cap B_{\frac{3}{2}}}a_j(e^{v_j}-1)\dd x
 &\leq \frac{C_n}{L}\int_{B_2}g_jv_j\dd x\leq\frac{C_n}{L}.
\end{aligned}
\]
Taking $L=\frac{\abs{\log E_j}}{2}$ for sufficiently large $j$ therefore proves
\begin{equation}\label{eq:g-strong}
 \norm{g_j-a_j}_{L^1\left(B_{\frac{3}{2}}\right)}
 \leq\frac{C_n}{\abs{\log E_j}}\to0.
\end{equation}
No subsequence or limit of the constants $a_j$ is needed.

We also need the endpoint Hessian estimate
\begin{equation}\label{eq:normalized-hessian}
 \norm{D^2v_j}_{L^1(B_1)}\leq C_nE_j.
\end{equation}
We prove \eqref{eq:normalized-hessian} as follows. The equation for $\frac{v_j}{E_j}$ has right hand side $g_j$. The Hessian operators are bounded on $L^2$ and of weak type $(1,1)$, as verified in the proof of Lemma~\ref{lem:constancy}. For each such operator $T$, splitting a non-negative source $g$ at height $t\geq1$ yields the distribution bound
\[
 \abs{\{\abs{Tg}>t\}\cap B_1}
 \leq \frac{C_n}{t}\int_{\{g>t\}}g
       +\frac{C_n}{t^2}\int_{\{g\leq t\}}g^2.
\]
Integration over $t\geq1$ bounds the local $L^1$ norm by
\[
 C_n\left(1+\int_{B_{3/2}} g+\int_{B_{3/2}} g\log^+g\right).
\]
First truncate and then smooth the source, with convergence in $L\log L$. The same layer-cake argument starting at an arbitrary height $a>0$ yields, for a tail $g-\min(g,N)$, a bound of the form
\[
 a\abs{B_1}+C_n\int_{B_{3/2}} \abs{g-\min(g,N)}\left[1+\log^+\frac{\abs{g-\min(g,N)}}{a}\right].
\]
Letting $N\to+\infty$ for fixed $a$, then $a\downarrow0$, proves local $L^1$ convergence of the singular integrals and identifies the distributional Hessian with an $L^1$ function. Extending the restriction of $g_j$ to $B_{\frac{3}{2}}$ by zero produces the Newtonian part, and the remaining harmonic part is controlled on $B_1$ by its $L^1$ norm. That norm is uniformly bounded by \eqref{eq:green-l1} after division by $E_j$ and the $L^1$ Green bound for the extended source. This proves \eqref{eq:normalized-hessian} from \eqref{eq:normalized-entropy}.

Set
\[
 q_j=\frac{\nabla v_j}{\sqrt{E_j}},\quad
 b_j=\frac{\nabla h_j}{\sqrt{E_j}}.
\]
On $B_1$ we have $q_j\in W^{1,1}\cap L^2$, uniformly bounded $L^2$ norm, and
\[
 \norm{q_j}_{L^1}+\norm{Dq_j}_{L^1}\to0,
 \quad \norm{b_j}_{L^\infty}\leq C_n.
\]
The normalized stationary identity is
\[
 \op{div}\left[\cT(q_j+b_j)+g_jI\right]=0.
\]
The tensor $\cT(b_j)$ is divergence free because $h_j$ is harmonic and smooth. Thus
\[
 \op{div}[\cT(q_j)-R_j]=0,
\]
where
\[
 R_j=-q_j\otimes b_j-b_j\otimes q_j+(q_j\cdot b_j)I-(g_j-a_j)I.
\]
The preceding estimates yield, for all sufficiently large $j$,
\[
 \norm{q_j}_{L^1(B_1)}+\norm{Dq_j}_{L^1(B_1)}\leq C_n\sqrt{E_j},
 \quad
 \norm{R_j}_{L^1(B_1)}
 \leq C_n\left(\sqrt{E_j}+\frac{1}{\abs{\log E_j}}\right).
\]
Corollary~\ref{cor:coercivity}, rescaled from $B_2,B_1$ to $B_1,B_{\frac{1}{2}}$, therefore yields
\begin{equation}\label{eq:quantitative-harmapprox}
 \frac{1}{E_j}\int_{B_{\frac{1}{2}}}\abs{\nabla v_j}^2\dd x
 =\norm{q_j}_{L^2\left(B_{\frac{1}{2}}\right)}^2
 \leq\frac{C_n}{\abs{\log E_j}}\to0.
\end{equation}
Here $E_j+\sqrt{E_j}$ is absorbed into $\frac{C_n}{\abs{\log E_j}}$ for small $E_j$. This proves the asserted convergence and its quantitative bound.
\end{proof}

\subsection{Excess decay estimate}

By  the estimate in Lemma \ref{lem:harmapprox}, we now prove the following excess decay estimate.
\begin{lem}\label{lem:energy-improvement}
For every $n\geq3$ there is $\varepsilon_n>0$ such that every stationary weak solution on $B_2$ with $\mathcal E_u(0,2)\leq\varepsilon_n$ satisfies
\[
 \mathcal E_u\left(0,\frac{1}{2}\right)\leq\frac{1}{8}\mathcal E_u(0,2).
\]
The radius ratio and contraction factor are fixed. The smallness threshold may depend on $n$.
\end{lem}

\begin{proof}
For a sequence with $E_j\to0$, use Lemma~\ref{lem:harmapprox} and its notation, and write
\[
 D_{h,j}=\int_{B_2}\abs{\nabla h_j}^2\dd x,\quad M_j=\int_{B_2}f_j\dd x.
\]
For every fixed $0<\theta\leq\frac{1}{2}$, centered subharmonicity of $\abs{\nabla h_j}^2$ on $B_2$ yields
\[
 \int_{B_\theta}\abs{\nabla h_j}^2\dd x\leq\left(\frac{\theta}{2}\right)^nD_{h,j}.
\]
The $v_j$ gradient energy on the whole open ball $B_{\frac{1}{2}}$ is $o(E_j)$, so Cauchy--Schwarz makes the mixed energy $o(E_j)$ there as well. Equations~\eqref{eq:constant-mass-bound}--\eqref{eq:g-strong} yield
\[
 \int_{B_\theta}f_j\dd x
 =E_j a_j\abs{B_\theta}+o(E_j)
 \leq\left(\frac{\theta}{2}\right)^n M_j+o(E_j).
\]
Since $D_{h,j}+M_j\leq E_j$, adding the two estimates yields
\[
 \int_{B_\theta}(\abs{\nabla u_j}^2+e^{u_j})\dd x
 \leq\left(\frac{\theta}{2}\right)^n E_j+o(E_j).
\]
Thus the exact scale factors yield
\[
 \limsup_j\frac{\mathcal E_{u_j}(0,\theta)}{\mathcal E_{u_j}(0,2)}
 \leq\left(\frac{\theta}{2}\right)^2.
\]
At $\theta=\frac{1}{2}$ the right side is $\frac{1}{16}$. If the claimed threshold did not exist, a sequence with $\mathcal E_{u_j}(0,2)\to0$ and energy ratio greater than $\frac{1}{8}$ would contradict this bound.
\end{proof}

\subsection{Regularity estimate and control on the singular set}
In this subsection we first finish the proof of Theorem \ref{thm:weak-epsilon}. This is almost standard now, because
iteration of the excess decay estimate at every interior center bounds the Newtonian potential.

\begin{proof}[Proof of Theorem~\ref{thm:weak-epsilon}]
Put $\sigma=\frac{1}{4}$. Scaling Lemma~\ref{lem:energy-improvement} yields
\[
 \mathcal E_u(x,\sigma R)\leq\frac{1}{8}\mathcal E_u(x,R)
\]
whenever $B_R(x)$ is in the domain and $\mathcal E_u(x,R)\leq\varepsilon_n$. For every $x\in B_{\frac{1}{2}}$, $B_1(x)\subset\subset B_2$ and
\[
 \mathcal E_u(x,1)\leq2^{n-2}\mathcal E_u(0,2).
\]
Choose $\varepsilon_n^*\leq2^{2-n}\varepsilon_n$. Iteration at all such centers, followed by comparison with neighboring discrete radii, yields
\begin{equation}\label{eq:improved-morrey}
 \int_{B_r(x)}(\abs{\nabla u}^2+e^u)\dd y
 \leq C_n E r^{n-2+\alpha},\quad
 x\in B_{\frac{1}{2}},\quad0<r\leq1,
\end{equation}
where $E=\int_{B_2}(\abs{\nabla u}^2+e^u)$ and $\alpha=\frac{\log(\frac{1}{8})}{\log(\frac{1}{4})}=\frac{3}{2}$.

Let $V$ be the Newtonian potential of $e^u\mathbf1_{B_1}$. Summing its integral over dyadic annuli centered at $x\in B_{\frac{1}{2}}$, using \eqref{eq:improved-morrey}, yields
\[
 0\leq V(x)\leq C_nE\sum_{k\geq0}2^{-k\alpha}+C_nE\leq C_nE.
\]
The distribution $H=u-V$ is harmonic on $B_1$, and therefore smooth there. It follows that $u$ is locally bounded above on $B_{\frac{1}{2}}$, so the weak equation and interior elliptic bootstrap make $u$ smooth on that ball. Since $V\geq0$, $e^H\leq e^u$. Subharmonicity of $e^H$ yields
\[
 \sup_{B_{\frac{1}{4}}}e^H\leq C_n\int_{B_1}e^u\dd x.
\]
Together with the bound for $V$, and after fixing the small energy threshold, this yields the quantitative estimate in the theorem.
\end{proof}
With this $\ep$-regularity in hand, the Hausdorff dimension estimate on the singular set is also standard. Here we give the proof for completeness.
\begin{proof}[Proof of Corollary \ref{coro:Hausdorff dim of singular set}]
Work first on a compact set contained in a slightly larger relatively compact open subset of a general domain, and set $w=\abs{\nabla u}^2+e^u\in L^1$ on that larger set. The density set
\[
 Z=\left\{x:\limsup_{r\downarrow0}r^{2-n}\int_{B_r(x)}w\dd y>0\right\}
\]
has zero $\mathcal H^{n-2}$ measure. For completeness, split $w$ into a bounded part and its $L^1$ tail. The bounded part has vanishing density at this exponent. A Vitali covering bounds the Hausdorff measure of each positive-density level by a constant times the tail integral. Letting the truncation height tend to infinity proves the assertion. At every point outside $Z$, a sufficiently small ball satisfies the epsilon hypothesis. Thus the solution is smooth near that point, and its relatively closed singular set is contained in $Z$.
\end{proof}

\section{Explicit solutions and limits of the hypotheses}\label{sec:calibrations}

In this section, we use some explicit solution families to test the contraction factor, the energy normalization, and the distinction between weak solutions and stationary weak solutions. These examples show the sharpness of various assumptions in our regularity theorems.

\subsection{The limiting contraction factor}

Small radial Dirichlet solutions show why the energy improvement needs a strict margin above its limiting ratio.

The contraction factor $\frac{1}{8}$ in Lemma~\ref{lem:energy-improvement} is larger than the limiting ratio $\frac{1}{16}$. That endpoint cannot replace $\frac{1}{8}$ for all sufficiently small solutions. To see this, let $z_\lambda$ be the small zero-boundary solution on $B_2$ of
\[
 -\Delta z_\lambda=\lambda e^{z_\lambda},\quad
 z_\lambda=0\quad\text{on }\partial B_2.
\]
For small $\lambda>0$, the Dirichlet Green operator is a contraction on a fixed small $L^\infty$ ball, yielding this smooth radial solution. Its first expansion is
\[
 z_\lambda(x)=\lambda\frac{4-\abs{x}^2}{2n}+O(\lambda^2)
 \quad\hbox{in }C^1(\overline B_2).
\]
Then $u_\lambda=\log\lambda+z_\lambda$ is a classical exact solution of the original equation. Writing $\omega_n=\abs{B_1}$, direct integration yields, at each fixed $0<r\leq2$,
\[
 \mathcal E_{u_\lambda}(0,r)
 =\omega_n r^2\left[
 \lambda+\lambda^2\left(\frac{2}{n}-\frac{(n-2)r^2}{2n(n+2)}\right)
 +O(\lambda^3)\right].
\]
Consequently
\[
 \frac{\mathcal E_{u_\lambda}\left(0,\frac{1}{2}\right)}{\mathcal E_{u_\lambda}(0,2)}
 =\frac{1}{16}\left[1+\lambda\frac{15(n-2)}{8n(n+2)}+O(\lambda^2)\right]
 >\frac{1}{16}
\]
for all sufficiently small $\lambda>0$, while the outer energy tends to zero.

\subsection{Singular cylinders and dimension dependence}

The following cylindrical solutions show why the signed variational energy cannot replace the positive energy, and why the scaling and dimension dependence matter.

The smallness condition concerns the sum of two positive terms. Replacing it by the signed variational energy would produce a false criterion: a codimension-three logarithmic cylinder has $\frac{1}{2}\abs{\nabla u}^2-e^u=0$ almost everywhere while remaining singular.

More generally, if $x=(y,z)\in\R^k\times\R^{n-k}$ with $3\leq k\leq n$, then
\[
 u(y,z)=\log(2(k-2))-2\log\abs{y}
\]
is a stationary $H^1_{\rm loc}$ weak solution and is singular on $\{y=0\}$. The weak-equation flux vanishes at the core. For $k>3$ the absolute stress flux vanishes. For $k=3$ its leading angular integral cancels and the remaining error tends to zero. At a ball centered on the singular plane, direct integration yields
\[
 r^{2-n}\int_{B_r}e^u\dd x=2\abs{\mathbb S^{n-1}},\quad
 r^{2-n}\int_{B_r}\abs{\nabla u}^2\dd x=\frac{4\abs{\mathbb S^{n-1}}}{k-2}.
\]
Thus these examples have a positive fixed-dimensional invariant energy and do not contradict the theorem. They also show why the scale factor cannot be omitted.

With the unnormalized Euclidean measure convention used here, a threshold independent of all dimensions is impossible. For the radial example $k=n$, its invariant energy is
\[
 \left(2+\frac{4}{n-2}\right)\abs{\mathbb S^{n-1}}\to0
 \quad\hbox{as }n\to+\infty,
\]
although the origin remains singular. The main theorem accordingly allows $\varepsilon_n$ to depend on $n$.

\subsection{A weak solution with a stationary stress defect}

Here we give a non-radial homogeneous family, showing that the weak equation alone does not imply the domain-variation identity.
Thus the stationary identity is independent of the weak equation. 

In $\R^3$, for $0<t<1$, the homogeneous function
\[
 w_t(x)=\log\left(2(1-t^2)\right)-2\log\left(\abs{x}-t x_3\right)
\]
belongs to $H^1_{\rm loc}$ and solves $-\Delta w_t=e^{w_t}$ distributionally. Direct differentiation verifies the equation away from zero, and its weak-equation flux is $O(r)$ at the origin. For the stationary stress $S_t=\cT(\nabla w_t)+e^{w_t}I$, direct spherical integration instead yields
\[
 \int_{\partial B_r}S_t\nu\dd \sigma
 =-\frac{8\pi}{t^2}\left(\log\frac{1+t}{1-t}-2t\right)e_3\ne0.
\]
Thus $\op{div} S_t$ is a non-zero point force. Constant extension in additional coordinates retains this defect on a codimension-three plane. The positive invariant energy in the original three dimensions is
\[
 \int_{B_1}(\abs{\nabla w_t}^2+e^{w_t})\dd x
 =\frac{16\pi}{t}\log\frac{1+t}{1-t}-8\pi\geq24\pi.
\]

\section{Blow-up analysis in arbitrary dimension}\label{sec:blowup}

In this section, we apply Theorem~\ref{thm:weak-epsilon} to extend
the three-dimensional blow-up analysis of Da Lio and Hyder
\cite[Theorem~1.3]{DaLioHyder} to arbitrary dimensions.
The main step is to prove that the Green part converges strongly
to zero when the mass tends to zero and the Dirichlet energy
remains bounded. This allows us to replace the smallness assumption
on the Dirichlet energy in Theorem \ref{thm:weak-epsilon} by an a priori bound, and then establish
a compactness alternative for sequences of stationary solutions.
Under an additional local Morrey bound on the Dirichlet energy,
we prove strong compactness and study tangent solutions and
the dimension of the singular set.

Throughout this section $s=n-2$ and
\[
 M_u(x,r)=r^{-s}\int_{B_r(x)}e^u\dd y,\quad
 D_u(x,r)=r^{-s}\int_{B_r(x)}\abs{\nabla u}^2\dd y.
\]
Thus $\mathcal E_u=M_u+D_u$. Stationary weak solutions are understood
in the sense of \eqref{equationLiouville} and
\eqref{eq:stationarity-definition}.

\subsection{\texorpdfstring{$L\log L$}{} estimate and two-dimensional slices}

We first establish an $L\log L$ estimate for $e^u$ in terms of
the mass and the Dirichlet energy. As in the proof of
Lemma~\ref{lem:harmapprox}, this estimate gives an $L^1$ bound
for the Hessian of the Green part. Here the estimate must be
independent of any lower bound for the mean of $u$, since this
mean may tend to $-\infty$ in the blow-up analysis.
\begin{lem}\label{lem:ba-entropy}
Let $u\in H^1(B_2)$ be a weak solution with $e^u\in L^1(B_2)$.
Then
\begin{equation}\label{eq:ba-entropy}
 \int_{B_{\frac{3}{2}}}e^u\log(2+e^u)\dd x
 \leq C_n\int_{B_2}(\abs{\nabla u}^2+e^u)\dd x.
\end{equation}
Let $h$ be the harmonic function with the boundary trace of $u$ and
put $v=u-h\in H^1_0(B_2)$. Then $v\geq0$,
\begin{equation}\label{eq:ba-green}
 \int_{B_2}(\abs{\nabla h}^2+\abs{\nabla v}^2)\dd x
 =\int_{B_2}\abs{\nabla u}^2\dd x,\quad
 \norm{v}_{W^{1,1}(B_2)}\leq C_n\int_{B_2}e^u.
\end{equation}
\end{lem}
This lemma implies that a uniform bound for the right hand side of \eqref{eq:ba-entropy} also controls
$D^2v$ in $L^1$ on each compact subset of $B_2$. 
\begin{proof}
Choose $\eta\in C_c^\infty(B_2)$ equal to one on $B_{\frac{3}{2}}$ and test the
equation with $\eta^2\min(u^+,L)$. The truncation justification is the
same as for \eqref{eq:fv-energy}. The resulting identity, Young's
inequality, and $(u^+)^2\leq C e^u$ yield
\[
 \int\eta^2 e^u\min(u^+,L)
 \leq C_n\int_{B_2}(\abs{\nabla u}^2+e^u).
\]
Let $L\to\infty$ and use $\log(2+e^u)\leq C+u^+$.
The Green assertions follow from orthogonality and
\eqref{eq:green-l1}. Finally apply the local $L\log L$ to $L^1$
Hessian estimate proved in \eqref{eq:normalized-hessian}, using a cutoff
on an intermediate ball and estimating the remaining harmonic function
by its $L^1$ norm. Rescaling and localization yield the statement on
every compact subset of $B_2$.
\end{proof}
The preceding Hessian bound alone does not exclude concentration
of the Dirichlet energy. To make use of the equation, we estimate
the weak-$L^2$ norm of the gradient on two-dimensional slices.
Combining this estimate with the two-dimensional
Lorentz--Sobolev inequality, we show that the energy on these
slices converges to zero in measure when the source tends to zero.
\begin{lem}\label{lem:ba-slice}
Write $x=(z,t)\in\R^s\times\R^2$. For $f\in L^1(\R^n)$ set
\[
 I_1f(x)=\int_{\R^n}\abs{x-y}^{1-n}\abs{f(y)}\dd y,
 \quad F(z)=\int_{\R^2}\abs{f(z,t)}\dd t.
\]
If $\mathcal M_s$ is the Hardy--Littlewood maximal operator on $\R^s$, then
\begin{equation}\label{eq:ba-sliceweak}
 \norm{I_1f(z,\cdot)}_{L^{2,\infty}(\R^2)}
 \leq C_n\mathcal M_sF(z)\quad\hbox{for almost every }z.
\end{equation}
Consequently, suppose $w_j\to0$ in $W^{1,1}_{\loc}(B_4)$,
$-\Delta w_j\to0$ in $L^1(B_4)$, and $(D^2w_j)$ is bounded in
$L^1(B_3)$. For $\eta\in C_c^\infty(B_2^2)$, the functions
\begin{equation}\label{eq:ba-slicevanish}
 z\mapsto\int_{\R^2}\eta(t)^2\abs{\nabla w_j(z,t)}^2\dd t
\end{equation}
tend to zero in measure on $B_2^s$.
\end{lem}
\begin{proof}
Fix $z$ with $m=\mathcal M_sF(z)<\infty$ and decompose $I_1f(z,t)$
as $\sum_{k\in\Z}a_k(t)$, restricting the source to
$2^{k-1}<\abs{z-\zeta}\leq2^k$ in $a_k$. For $p=\frac{3}{2}$, the maximal
bound and Minkowski's inequality yield
\[
 \norm{a_k}_\infty\leq C_nm2^{-k},\quad
 \norm{a_k}_{L^p(\R^2)}\leq C_nm2^{k(\frac{2}{p}-1)}.
\]
Indeed, the $L^p(\R^2)$ norm of
$(a^2+\abs{t}^2)^{-\frac{n-1}{2}}$ is $C_{n,p}a^{\frac{2}{p}-(n-1)}$.
For $\lambda>0$ choose $2^{k_0}$ comparable to $\frac{m}{\lambda}$, with
a sufficiently large fixed factor. The sum over $k>k_0$ is bounded
by $\frac{\lambda}{2}$, whereas Chebyshev's inequality for the remaining sum yields
\[
 \abs{\{t:I_1f(z,t)>\lambda\}}
 \leq C_n\lambda^{-p}m^p2^{k_0(2-p)}
 \leq C_nm^2\lambda^{-2}.
\]
The case $m=0$ is immediate. This proves \eqref{eq:ba-sliceweak}.

For the second assertion let $f_j=-\Delta w_j$, and let $N_j$ be the
Newtonian potential of $f_j\mathbf1_{B_4}$. Local integrability of the
potential kernel and its gradient implies that $N_j\to0$ in
$W^{1,1}_{\loc}(\R^n)$, and
$\abs{\nabla N_j}\leq C_n I_1(f_j\mathbf1_{B_4})$.
The difference $w_j-N_j$ is harmonic on $B_4$ and tends to zero in
local $L^1$, so its gradient tends to zero uniformly on $B_3$.
Equation~\eqref{eq:ba-sliceweak} and the
weak $(1,1)$ maximal estimate imply that
\[
 a_j(z):=\norm{\eta\nabla w_j(z,\cdot)}_{L^{2,\infty}(\R^2)}
 \to0\quad\hbox{in measure on }B_2^s.
\]
The two-dimensional Lorentz--Sobolev inequality yields
\[
 b_j(z):=\norm{\eta\nabla w_j(z,\cdot)}_{L^{2,1}(\R^2)}
 \leq C_\eta\int_{B_2^2}(\abs{D_t\nabla w_j}+\abs{\nabla w_j})\dd t,
 \quad \sup_j\int_{B_2^s}b_j<\infty.
\]
Here $B_2^s\times B_2^2\subset\subset B_3$. The Sobolev inequality follows
from coarea and the planar isoperimetric inequality, first for smooth
compactly supported functions and then by approximation in $W^{1,1}$.
Lorentz duality bounds \eqref{eq:ba-slicevanish} by $C a_jb_j$.
For $\epsilon,\delta>0$ the measure where this product exceeds
$\epsilon$ is at most
\[
 \abs{\{a_j>\delta\}}+C\delta\epsilon^{-1}\int b_j.
\]
Letting $j\to\infty$ and then $\delta\downarrow0$ proves
\eqref{eq:ba-slicevanish}.
\end{proof}

\subsection{Elimination of a source-free concentration measure}

We now use the slicing estimate to rule out concentration of
the gradient energy. The main difficulty is that the Hessians
are only bounded in $L^1$, so
Theorem~\ref{thm:abstractcompact} cannot be applied directly.
Under a lower density bound for the concentration measure,
we can repeat the dimension reduction argument of
Section~\ref{sec:compactness}. The final two-dimensional step
is then provided by Lemma~\ref{lem:ba-slice}.

\begin{lem}\label{lem:ba-defect}
Let $G\subset\subset\R^n$ be open and $q_j=\nabla w_j$, where
$w_j\in W^{2,1}_{\loc}(G)$, $w_j\to0$ in $W^{1,1}_{\loc}(G)$,
and $-\Delta w_j\to0$ in $L^1_{\loc}(G)$. Suppose $q_j\in L^2_{\loc}(G)$ and
\[
 \sup_j\int_K(\abs{q_j}^2+\abs{Dq_j})<\infty
 \quad(K\subset\subset G),\quad
 \op{div}(\mathcal T(q_j)-R_j)=0,\quad
 R_j\to0\text{ in }L^1_{\loc}(G),
\]
where $R_j\in L^1_{\loc}(G;\R^{n\times n})$ and
$\mathcal T(q)=q\otimes q-\frac{1}{2}\abs{q}^2 I$.
Take weak measure limits $q_j\otimes q_j\dd x\rightharpoonup A$ and
$\abs{q_j}^2\dd x\rightharpoonup\mu=\tr A$.
Assume that on every $K\subset\subset G$ there are $a_K,r_K>0$ such that
\begin{equation}\label{eq:ba-lowerdensity}
 \mu(B_r(x))\geq a_Kr^s
 \quad(x\in K\cap\supp\mu,\ 0<r<r_K).
\end{equation}
Then $\mu=0$.
\end{lem}
\begin{proof}
Write $A=P\mu$, with $P\geq0$ and $\tr P=1$ $\mu$-almost everywhere.
Kernel statements for a matrix measure are understood in this sense,
as in the conventions of the introduction. The limit satisfies $A\geq0$ and
$\op{div}(A-\frac{1}{2}\mu I)=0$. The local version of
Lemma~\ref{lem:monotone} yields a non-decreasing density and, using
\eqref{eq:ba-lowerdensity},
\begin{equation}\label{eq:ba-densitylimit}
 0<\Theta(x):=\lim_{r\downarrow0}r^{-s}\mu(B_r(x))<\infty
 \quad(x\in\supp\mu).
\end{equation}
There is also a uniform upper bound $\mu(B_r(x))\leq C_Kr^s$ for
small balls centered in each fixed compact subset, by comparing with a
fixed larger radius in the monotonicity formula.

Pass to a further subsequence with $\abs{Dq_j}\dd x\rightharpoonup\beta$.
At $\mu$-almost every $x$ the differentiation theorem for Radon measures yields
\begin{equation}\label{eq:ba-betadensity}
 \limsup_{r\downarrow0}\frac{\beta(B_r(x))}{\mu(B_r(x))}<\infty.
\end{equation}
At such a point choose $r_\ell\downarrow0$ and define the rescaled
measures on Borel sets $E$ by
\[
 A^{(\ell)}(E)=r_\ell^{-s}A(x+r_\ell E),\quad
 \mu^{(\ell)}(E)=r_\ell^{-s}\mu(x+r_\ell E).
\]
The upper growth bound yields local weak measure compactness. Pass to
a limit $(\widehat A,\widehat\mu)$ and write
$\widehat A=\widehat P\widehat\mu$. At each continuity radius of
$\widehat\mu$, weak convergence and \eqref{eq:ba-densitylimit} imply
$\widehat\mu(B_R)=\Theta(x)R^s$. Squeezing between continuity radii
extends this identity to every $R>0$; in particular, every centered
sphere has zero $\widehat\mu$ mass.
Equality in the monotonicity formula implies
$\widehat P(y)y=0$ for $\widehat\mu$-almost every $y$. Testing stationarity
with $X(y)=\phi(y)y$, where $\phi\in C_c^\infty(\R^n)$, yields
$\int_{\R^n}(s\phi+y\cdot\nabla\phi)\dd\widehat\mu=0$.
Thus $\widehat\mu$ is homogeneous of degree $s$.

To iterate this construction, we choose a sequence realizing the tangent
pair and retaining the local Hessian bound. Relabel the radii so that
$B_{4\ell r_\ell}(x)\subset\subset G$. For each fixed $r_\ell$, choose $j(\ell)$
late enough and set
\[
 \widehat w_\ell(y)=w_{j(\ell)}(x+r_\ell y)
     -\fint_{B_1}w_{j(\ell)}(x+r_\ell z)\dd z,\quad
 \widehat q_\ell=r_\ell q_{j(\ell)}(x+r_\ell\cdot),\quad
 \widehat R_\ell=r_\ell^2R_{j(\ell)}(x+r_\ell\cdot).
\]
The $W^{1,1}$, source, and tensor errors can be made arbitrarily small
on an increasing finite list of balls, since the scale is fixed before
the index is chosen. Poincar\'e's inequality handles the subtracted mean.
Require covariance approximation on the first $\ell$ members of a
countable dense family of compactly supported test functions. Local bounds for the rescaled energy and Hessian are obtained by the
following cutoff conditions.
For each integer $1\leq m\leq\ell$, fix $0\leq\chi_m\leq1$, equal
to one on $B_m$ and supported in $B_{2m}$, and choose $j(\ell)$ so that
\[
\begin{aligned}
 \int\chi_m\abs{\widehat q_\ell}^2\dd y
 &\leq r_\ell^{-s}\int\chi_m\left(\frac{y-x}{r_\ell}\right)\dd\mu(y)
       +\ell^{-1},\\
 \int\chi_m\abs{D\widehat q_\ell}\dd y
 &\leq r_\ell^{-s}\int\chi_m\left(\frac{y-x}{r_\ell}\right)\dd\beta(y)
       +\ell^{-1}.
\end{aligned}
\]
These are finitely many valid requirements on the original weakly
convergent measures at each fixed scale. Equations
\eqref{eq:ba-densitylimit}--\eqref{eq:ba-betadensity} imply, for each fixed $m$,
\[
 \limsup_{\ell\to\infty}\int_{B_m}
    (\abs{\widehat q_\ell}^2+\abs{D\widehat q_\ell})\dd y
 \leq C_xm^s.
\]
These local bounds extend convergence from the countable test family
to all continuous compactly supported functions. Hence the sequence
realizes $(\widehat A,\widehat\mu)$ and has locally bounded Hessians.

The lower density bound also passes to every point in the tangent
support. Fix a compact neighborhood $K$ of $x$ contained in $G$.
If $y\in\supp\widehat\mu$, choose $y_\ell\in\supp\mu^{(\ell)}$
with $y_\ell\to y$. For each fixed $\rho>0$ and large $\ell$,
the corresponding original centers lie in $K$ and
$B_{\frac{\rho}{4}}(y_\ell)\subset\subset\overline B_{\frac{\rho}{2}}(y)$. Hence
\[
 \mu^{(\ell)}(\overline B_{\frac{\rho}{2}}(y))
 \geq a_K(\frac{\rho}{4})^s.
\]
The closed-set inequality for weak convergence yields
$\widehat\mu(B_\rho(y))\geq a_K(\frac{\rho}{4})^s$.
Thus the lower density condition is retained, with a fixed smaller constant.

We now repeat this procedure to gain constant zero directions. Suppose
a non-zero tangent pair $A=P\mu$ has $P(y)y=0$ and $P(y)v=0$ for $v$ in a
$k$-dimensional space $V$, where $k<s$. The upper $s$-growth bound
implies $\mu(V)=0$: cover a bounded part of $V$ by $O(r^{-k})$
balls and send $r\downarrow0$. We can therefore choose
$x\in\supp\mu\backslash V$ satisfying \eqref{eq:ba-betadensity}
for the current realizing sequence. Put
$a=\abs{\operatorname{proj}_{V^\perp}x}$ and
$e=\frac{\operatorname{proj}_{V^\perp}x}{a}$. Positivity and the kernel conditions yield
\[
 e^TP(y)e
 =a^{-2}(x-y)^TP(y)(x-y)
 \leq a^{-2}\abs{x-y}^2
 \quad\hbox{for }\mu\hbox{-almost every }y.
\]
On rescaling about $x$ this extra factor is $O(r_\ell^2)$ on each
fixed ball. Every new tangent therefore annihilates $e$ as well as
$V$, and again has a radial kernel by \eqref{eq:ba-densitylimit}.
All analytic bounds and the lower density condition are retained as
above. After $s$ repetitions we have a non-zero global pair with a
realizing sequence for which
\[
 \int_K\abs{\operatorname{proj}_Vq_j}^2\to0,
 \quad \dim V=s,\quad K\subset\subset\R^n.
\]

Rotate so $V=\R^s\times\{0\}$. Choose a compactly supported transverse
cutoff $\eta$ for which the projection of $\eta^2\mu$ is non-zero, and set
$F_j(z)=\int\eta^2\abs{q_j}^2\dd t$.
The computation in Lemma~\ref{lem:terminal} yields
\[
 \nabla_zF_j=\op{div}_zE_j+b_j,\quad
 E_j\to0,\quad b_j\to0\quad\hbox{locally in }L^1.
\]
Thus its distributional limit is a positive constant $\theta$ on
$\R^s$. Positivity follows by testing in a cylinder carrying positive
$\eta^2\mu$ mass. Lemma~\ref{lem:constancy} implies
$F_j\to\theta>0$ in measure on interior balls. After a fixed dilation
and localization, Lemma~\ref{lem:ba-slice} instead yields $F_j\to0$
in measure. This contradiction proves the lemma.
\end{proof}

\subsection{Small mass with bounded Dirichlet energy}

We next consider a sequence of stationary solutions whose masses
tend to zero and whose Dirichlet energies remain bounded.
The main point is to prove strong $H^1_{\loc}$ convergence
of the Green parts to zero. To apply Lemma~\ref{lem:ba-defect},
we use the small-energy regularity theorem to verify the required
lower density bound for any non-zero concentration measure.
Once the Green parts converge strongly, the same regularity
theorem implies that the solutions tend locally uniformly
to $-\infty$.

\begin{lem}\label{lem:ba-vanishing}
Let $u_j\in H^1(B_2)$ be stationary weak solutions with
$e^{u_j}\in L^1(B_2)$ and
\begin{equation}\label{eq:ba-vanishing-assumption}
 \int_{B_2}e^{u_j}\dd x\to0,\quad
 \sup_j\int_{B_2}\abs{\nabla u_j}^2\dd x\leq\Lambda<\infty.
\end{equation}
For the Green decompositions $u_j=h_j+v_j$ one has
$v_j\to0$ strongly in $H^1_{\loc}(B_2)$. Moreover,
\[
 \operatorname*{ess\,sup}_K u_j\to-\infty
 \quad\hbox{for every }K\subset\subset B_2.
\]
In particular $u_j$ is smooth near $K$ for all sufficiently large $j$.
\end{lem}
\begin{proof}
By Lemma~\ref{lem:ba-entropy}, $v_j\to0$ in $W^{1,1}(B_2)$,
the gradients are bounded in $L^2$, and the Hessians are locally
bounded in $L^1$. Harmonic estimates yield locally uniform bounds for
every derivative of $\nabla h_j$. After a subsequence
$\nabla h_j\to b$ smoothly locally, where $b$ is a harmonic gradient.
Subtract the harmonic stress from stationarity and put $q_j=\nabla v_j$:
\[
 \op{div}(\mathcal T(q_j)-R_j)=0,\quad
 R_j=-q_j\otimes\nabla h_j-\nabla h_j\otimes q_j
       +(q_j\cdot\nabla h_j)I-e^{u_j}I.
\]
Here $R_j\to0$ locally in $L^1$.

Let $A$ be a covariance limit of $q_j$ and $\mu=\tr A$.
We verify \eqref{eq:ba-lowerdensity}. On a fixed interior compact set,
$b$ is bounded, so $r^{-s}\int_{B_r(x)}\abs{b}^2\leq C r^2$.
If $x\in\supp\mu$ and $\mu(B_r(x))<a_nr^s$ for a sufficiently
small fixed dimensional $a_n$ and a sufficiently small $r$, choose a
continuity radius $\rho\in(\frac{r}{2},r)$. On $B_\rho(x)$ the limiting
positive energy is
\[
 \lim_j\int_{B_\rho(x)}(\abs{\nabla u_j}^2+e^{u_j})
 =\int_{B_\rho(x)}\abs{b}^2+\mu(B_\rho(x)).
\]
The mixed terms vanish by local $L^1$ convergence of $q_j$.
Taking $a_n$ and then the uniform upper bound for $r$ small enough
makes $\mathcal E_{u_j}(x,\rho)<\varepsilon_n$ for large $j$.
The scaled form of Theorem~\ref{thm:weak-epsilon} yields
\[
 \sup_{B_{\frac{\rho}{8}}(x)}e^{u_j}
 \leq C_n\rho^{-n}\int_{B_2}e^{u_j}\to0.
\]
Interior estimates for $-\Delta v_j=e^{u_j}$, together with
$v_j\to0$ in $L^1$, imply $\nabla v_j\to0$ uniformly on a smaller
ball. This contradicts $x\in\supp\mu$. The required lower density
bound follows, and Lemma~\ref{lem:ba-defect} yields $\mu=0$.
Every subsequence has a further subsequence whose covariance limit
vanishes. Therefore $\nabla v_j\to0$ locally in $L^2$. Local Poincar\'e and
$L^1$ convergence yield the asserted $H^1$ convergence.

Finally cover a compact set by finitely many small fixed balls on which
the harmonic Dirichlet energies are below the small-energy threshold.
The Green energies and source masses tend to zero on these balls.
Theorem~\ref{thm:weak-epsilon} now bounds $\sup_K e^{u_j}$ by a fixed
multiple of $\int_{B_2}e^{u_j}$, which tends to zero.
\end{proof}

\begin{thm}[Small mass at bounded energy]\label{thm:ba-smallmass}
For every $n\geq3$ and $0\leq\Lambda<\infty$ there is
$\delta(n,\Lambda)>0$ such that every stationary weak solution
$u\in H^1(B_2)$ with $e^u\in L^1(B_2)$ satisfying
\begin{equation}\label{eq:ba-smallmass}
 \int_{B_2}\abs{\nabla u}^2\leq\Lambda,\quad
 \int_{B_2}e^u\leq\delta(n,\Lambda)
\end{equation}
is smooth on $B_1$ and satisfies
\begin{equation}\label{eq:ba-smallmass-bound}
 \sup_{B_1}e^u\leq C_n\int_{B_2}e^u.
\end{equation}

    The constant in \eqref{eq:ba-smallmass-bound} depends only on the
dimension; the admissible mass threshold $\delta(n,\Lambda)$ may depend on $\Lambda$.
Equivalently, on $B_{2r}(x)$ one may assume
$D_u(x,2r)\leq\Lambda$ and $M_u(x,2r)\leq\delta_*(n,\Lambda)$,
obtaining $\sup_{B_r(x)}e^u\leq C_nr^{-n}\int_{B_{2r}(x)}e^u$.
\end{thm}
\begin{proof}
Lemma~\ref{lem:ba-vanishing}, by contradiction, yields a threshold
for which $u\leq0$ almost everywhere on $B_{\frac{3}{2}}$. Solve
$-\Delta V=e^u$ there with zero boundary values. Since $e^u\leq1$,
comparison with the torsion function yields $0\leq V\leq C_n$.
The function $H=u-V$ is harmonic and $e^H\leq e^u$. Hence
\[
 \sup_{B_1}e^u\leq e^{C_n}\sup_{B_1}e^H
 \leq C_n\int_{B_{\frac{3}{2}}}e^u.
\]
Elliptic regularity yields smoothness. 

For the last formulation put
$\widetilde u(y)=u(x+ry)+2\log r$. The exact scaling is
\[
 \int_{B_2}\abs{\nabla\widetilde u}^2=2^sD_u(x,2r),\quad
 \int_{B_2}e^{\widetilde u}=2^sM_u(x,2r).
\]
Thus one may take $\delta_*(n,\Lambda)=2^{-s}\delta(n,2^s\Lambda)$.
Rescaling \eqref{eq:ba-smallmass-bound} then yields the stated pointwise estimate.
\end{proof}

\subsection{The compactness alternative}

We now establish the compactness alternative for sequences with
locally bounded mass and Dirichlet energy. By Poincar\'e's
inequality, the first issue is whether the means remain bounded
or tend to $-\infty$. The $L\log L$ estimate gives strong local
$L^1$ convergence of the sources in either case.
We then use the small-mass regularity theorem to obtain locally
uniform convergence to $-\infty$ in the latter case, and smooth
convergence away from a concentration set in the former.

\begin{thm}[Compactness alternative]\label{thm:ba-dichotomy}
Let $\Omega\subset\subset\R^n$ be a connected open set and let $u_j$ be stationary
weak solutions satisfying
\begin{equation}\label{eq:ba-localbound}
 \sup_j\int_K(\abs{\nabla u_j}^2+e^{u_j})<\infty
 \quad\hbox{for every }K\subset\subset\Omega.
\end{equation}
After passage to a subsequence, exactly one of the following occurs.
\begin{enumerate}[label=$(\theenumi)$]
\item $u_j\to-\infty$ locally uniformly in the essential-supremum
sense. Fix a ball $B_*\subset\subset\Omega$ and let $c_j=\fint_{B_*}u_j$.
Then $c_j\to-\infty$ and $u_j-c_j\to h$ in
$C^\infty_{\loc}(\Omega)$ for a harmonic function $h$.
\item There are a weak solution $u\in H^1_{\loc}(\Omega)$ and a
relatively closed set $\Sigma$ with $\mathcal H^s(\Sigma)=0$ such that
\begin{align}
 u_j&\rightharpoonup u&&\text{in }H^1_{\loc}(\Omega),\label{eq:ba-weakconv}\\
 u_j&\to u&&\text{in }W^{1,p}_{\loc}(\Omega),\quad1\leq p<2,\label{eq:ba-subcriticalconv}\\
 e^{u_j}&\to e^u&&\text{in }L^1_{\loc}(\Omega),\label{eq:ba-sourceconv}\\
 u_j&\to u&&\text{in }C^\infty_{\loc}(\Omega\backslash\Sigma).\label{eq:ba-smoothconv}
\end{align}
Here $\Sigma$ can be taken to be the actual concentration set of the
chosen subsequence, namely the points $x$ such that every interior
ball $B_r(x)$ satisfies
$\limsup_j\operatorname*{ess\,sup}_{B_r(x)}u_j=+\infty$.
There is a positive semidefinite matrix measure $A$ supported on
$\Sigma$ with
\begin{equation}\label{eq:ba-covariance}
 \nabla u_j\otimes\nabla u_j\dd x
 \rightharpoonup\nabla u\otimes\nabla u\dd x+A,
\end{equation}
and the precise limiting stationarity identity is
\begin{equation}\label{eq:ba-defect-stress}
 \op{div}\left(\mathcal T(\nabla u)+e^uI
          +A-\frac{1}{2}(\tr A)I\right)=0.
\end{equation}
\end{enumerate}
In the first alternative the $C^\infty$ assertion concerns each compact
set for sufficiently large indices, when the solutions there are smooth.
In the second alternative no vanishing of $A$ is asserted without
an additional argument.
\end{thm}
\begin{proof}
By Jensen inequality, $c_j$ is bounded above. Poincar\'e's inequality on connected
relatively compact Lipschitz subdomains containing $B_*$ shows that
$u_j-c_j$ is bounded in local $H^1$. Such subdomains can cover every
prescribed compact set, since $\Omega$ is connected. Rellich compactness
and a diagonal subsequence yield local $L^2$ and almost-everywhere
convergence of $u_j-c_j$. Passing to a further subsequence, either
$c_j\to-\infty$ or $c_j$ has a finite limit.
Lemma~\ref{lem:ba-entropy} yields local uniform integrability of
$e^{u_j}$, independent of the negative means.

If $c_j\to-\infty$, $e^{u_j}\to0$ in local $L^1$ by Vitali's theorem.
Lemma~\ref{lem:ba-vanishing} implies $u_j\to-\infty$ locally uniformly
in the essential-supremum sense.
The normalized functions $u_j-c_j$ have bounded local $L^2$ norm, and $\Delta(u_j-c_j)$ converges  uniformly
to $0$. Interior Poisson estimates first yield local
$C^{1,\alpha}$ compactness, for $0<\alpha<1$. Then
$e^{u_j}=e^{c_j}e^{u_j-c_j}$ and elliptic bootstrap yield convergence
in every local $C^m$ norm to a harmonic function.

If $c_j$ has a finite limit, $u_j\rightharpoonup u$ locally in $H^1$ and
$u_j\to u$ in local $L^2$ and almost everywhere. Vitali convergence theorem yields
\eqref{eq:ba-sourceconv}, so the weak equation passes to the limit.
Extract $\abs{\nabla u_j}^2\dd x\rightharpoonup\mu_2$ and define
\[
 Z_1=\left\{x:\limsup_{r\downarrow0}r^{-s}\int_{B_r(x)}e^u>0\right\},
 \quad
 Z_2=\left\{x:\liminf_{r\downarrow0}r^{-s}\mu_2(B_r(x))=\infty\right\}.
\]
Both sets have zero $\mathcal H^s$ measure. For $Z_1$ use the bounded
part and $L^1$ tail argument at the end of
Section~\ref{sec:liouville}. For $Z_2$, a Vitali covering of each
compact portion where $\mu_2(B_r(x))\geq Nr^s$ at all sufficiently
small radii yields a bound $\frac{C\mu_2(K')}{N}$; let $N\to\infty$.

If $x\notin Z_1\cup Z_2$, there are continuity radii $r_i\downarrow0$ with
\[
 \sup_i r_i^{-s}\mu_2(B_{r_i}(x))<\infty,\quad
 r_i^{-s}\int_{B_{r_i}(x)}e^u\to0.
\]
To obtain continuity radii,
shrink preliminary radii by a factor between $\frac{1}{2}$ and $1$; the
fixed loss in the bounds is harmless. For each fixed $i$, the corresponding
quantities for $u_j$ converge as $j\to\infty$. Rescale to $B_2$,
first fix $i$ large enough for Theorem~\ref{thm:ba-smallmass}, and
then take $j$ large. This yields a uniform upper bound near $x$.
Consequently $\Sigma\subset Z_1\cup Z_2$ and
$\mathcal H^s(\Sigma)=0$. The concentration set $\Sigma$ is relatively closed because its
complement is open by definition.

On that open complement, the local upper bounds and the $L^2$ bounds
for $u_j$ imply \eqref{eq:ba-smoothconv} by elliptic estimates.
The gradients therefore converge almost everywhere, and their local
$L^2$ bounds yield strong $L^p$ convergence for all $p<2$.
Weak lower semicontinuity of squared scalar projections defines a
positive semidefinite defect $A$ in \eqref{eq:ba-covariance}.
Smooth convergence off $\Sigma$ locates its support there. Passing to
the stationary identity, using the strong convergence of the sources,
proves \eqref{eq:ba-defect-stress}.
\end{proof}

\subsection{Strong compactness under a Morrey bound}

In the finite-limit alternative, it remains to determine whether
the gradients converge strongly in $L^2_{\loc}$.
The possible defect is supported on the concentration set
$\Sigma$, which has zero $\mathcal H^{n-2}$ measure.
A uniform local Morrey bound on the Dirichlet energy implies
that the limiting energy measure gives no mass to $\Sigma$.
This proves strong convergence and allows us to pass
the stationary condition to the limit.

\begin{thm}[Strong compactness under a Morrey bound]\label{thm:ba-morrey}
In the finite-limit alternative of Theorem~\ref{thm:ba-dichotomy}, suppose that
for every open set $D\subset\subset\Omega$ there is $\Lambda_D<\infty$ such that
\begin{equation}\label{eq:ba-morrey}
 r^{2-n}\int_{B_r(x)}\abs{\nabla u_j}^2\dd y\leq\Lambda_D
 \quad(B_r(x)\subset\subset D,\ j\geq1).
\end{equation}
Then $u_j\to u$ strongly in $H^1_{\rm loc}(\Omega)$, the limit $u$ is
stationary, and the concentration set of the selected sequence equals
$\Sing(u)$.  The convergence is $C^m_{\rm loc}(\Omega\backslash\Sing(u))$ for
every $m\geq0$.
\end{thm}
\begin{proof}
Write $s=n-2$, and let $\mu_2$ be the weak limit of
$\abs{\nabla u_j}^2\dd x$.  On balls compactly contained in $D$,
\eqref{eq:ba-morrey} passes to the limit and yields
$\mu_2(B_r(x))\leq\Lambda_Dr^s$.  Therefore $\mu_2$ charges no set of zero
$\mathcal H^s$ measure: cover such a set by balls and sum this bound.
Theorem~\ref{thm:ba-dichotomy} yields $\mathcal H^s(\Sigma)=0$ and smooth
convergence off $\Sigma$.  The positive gradient defect is supported on
$\Sigma$, hence vanishes. For each non-negative $\chi\in C_c^\infty(\Omega)$,
weak convergence of the energy measures and of the gradients yields
\[
 \int\chi\abs{\nabla u_j-\nabla u}^2\dd x
 =\int\chi\abs{\nabla u_j}^2\dd x
  -2\int\chi\nabla u_j\cdot\nabla u\dd x
  +\int\chi\abs{\nabla u}^2\dd x\to0.
\]
Taking $\chi=1$ near a prescribed compact set, and using the local
$L^2$ convergence of the functions, proves strong $H^1_{\rm loc}$ convergence.
The stresses now converge in $L^1_{\rm loc}$, so their divergence equation
passes to $u$.

Smooth convergence outside $\Sigma$ implies $\Sing(u)\subset\Sigma$.  If $u$ is smooth near $x$, choose $r>0$
so small that $\mathcal E_u(x,r)<\frac{\varepsilon_n}{2}$.  Strong gradient
convergence and strong $L^1$ convergence of the exponential imply
$\mathcal E_{u_j}(x,r)<\varepsilon_n$ for large $j$.  Theorem~\ref{thm:weak-epsilon}
makes $x$ a regular point of the sequence, proving the reverse inclusion.
Covering each compact subset of $\Omega\backslash\Sing(u)$ by finitely
many such balls and applying interior elliptic estimates proves the
asserted smooth convergence.
\end{proof}

\subsection{Monotonicity and tangent solutions}
We next study blow-up limits at singular points under the local
Morrey bound assumption. The strong compactness result above gives
convergence of the rescaled solutions. To identify the
homogeneity of their limits, we use the following monotonicity
formula, which follows from the equation and the stationary
condition. Its dissipation term measures the failure of
logarithmic homogeneity.

\begin{prop}[Monotonicity]\label{prop:ba-weiss}
For a stationary weak solution, define at almost every admissible radius
\begin{equation}\label{eq:ba-weiss}
 W_u(x,r)=r^{2-n}\int_{B_r(x)}
 \left(\frac{1}{2}\abs{\nabla u}^2-e^u\right)\dd y
 +2r^{1-n}\int_{\partial B_r(x)}(u+2\log r)\dd\sigma.
\end{equation}
It has a locally absolutely continuous representative on $r>0$, and
\begin{equation}\label{eq:ba-weiss-dissipation}
 W_u(x,R)-W_u(x,r)
 =\int_r^R t^{2-n}\int_{\partial B_t(x)}
 \left(\partial_\nu u+\frac{2}{t}\right)^2\dd\sigma\dd t\geq0.
\end{equation}
This monotonicity concerns $W_u$, not the positive energy $\mathcal E_u$.
\end{prop}
\begin{proof}
Translate $x$ to zero and put
$D(r)=\int_{B_r}\abs{\nabla u}^2$ and $F(r)=\int_{B_r}e^u$.
Radial vector fields in the stationary identity yield, for almost every $r$,
\[
 -\frac{n-2}{2}D(r)+nF(r)
 =r\int_{\partial B_r}
 \left(\abs{\partial_\nu u}^2-\frac{1}{2}\abs{\nabla u}^2+e^u\right)\dd\sigma.
\]
Consequently
\[
 \frac{\dd}{\dd r}\left[r^{2-n}\left(\frac{1}{2}D(r)-F(r)\right)\right]
 =r^{2-n}\int_{\partial B_r}\abs{\partial_\nu u}^2\dd\sigma
 -2r^{1-n}F(r).
\]
To check the flux at the center, put
$I(r)=\int_{\partial B_r}\partial_\nu u\dd\sigma$ for almost every $r$.
Fix an admissible $R_0$ and take a smooth radial test function $\eta(r)$
supported in $[0,R_0)$ and constant near zero. The weak equation and
one-dimensional integration by parts yield
\[
 \int_0^{R_0}\eta'(r)I(r)\dd r
 =\int_0^{R_0}\eta(r)F'(r)\dd r
 =-\int_0^{R_0}\eta'(r)F(r)\dd r.
\]
Here $F(0)=0$. Every smooth function compactly supported in $(0,R_0)$ is the
derivative of such a radial test function, whose value at zero is unrestricted.
Thus $I(r)=-F(r)$ almost everywhere, with no additional flux constant.
Differentiating the second term of \eqref{eq:ba-weiss} yields
$-2r^{1-n}F(r)+\frac{4\abs{\mathbb S^{n-1}}}{r}$.
Adding the two identities completes the square.  All terms are integrable
on annuli bounded away from zero; radial Sobolev slicing justifies the
differentiations and integration proves \eqref{eq:ba-weiss-dissipation}.
\end{proof}

\begin{thm}[Tangent solutions and singular mass density]\label{thm:ba-tangents}
Let $u$ be a stationary weak solution on an open set $\Omega\subset\subset\R^n$.
Suppose that for each open $D\subset\subset\Omega$ there is $\Lambda_D<\infty$
such that $D_u(z,\rho)\leq\Lambda_D$ whenever $B_\rho(z)\subset\subset D$.
If $x\in\Sing(u)$ and $r_j\downarrow0$, a subsequence of
\[
 u_{x,r_j}(y)=u(x+r_jy)+2\log r_j
\]
converges strongly in $H^1_{\rm loc}(\R^n)$ to a stationary weak solution $U$,
and $e^{u_{x,r_j}}\to e^U$ strongly in $L^1_{\rm loc}(\R^n)$.
The origin is singular for $U$, and
\begin{equation}\label{eq:ba-homogeneity}
 U(\lambda y)=U(y)-2\log\lambda
 \quad(\lambda>0)
\end{equation}
as an identity of locally integrable functions.  Equivalently,
\[
 U(y)=-2\log\abs{y}+\psi\left(\frac{y}{\abs{y}}\right),\quad
 \psi\in H^1(\mathbb S^{n-1}),\quad e^\psi\in L^1(\mathbb S^{n-1}),
\]
where, in the weak sense on the sphere,
\begin{equation}\label{eq:ba-angular}
 -\Delta_{\mathbb S^{n-1}}\psi+2(n-2)=e^\psi.
\end{equation}
In particular, without requiring uniqueness of $U$,
\begin{equation}\label{eq:ba-fixed-mass}
 \lim_{r\downarrow0}r^{2-n}\int_{B_r(x)}e^u\dd y
 =2\abs{\mathbb S^{n-1}}.
\end{equation}
\end{thm}
\begin{proof}
Fix a relatively compact neighborhood of $x$ and a Morrey constant $\Lambda$
there.  For $B_{2r}(z)$ in this neighborhood, testing the equation against a
cutoff equal to one on $B_r(z)$ yields
\begin{equation}\label{eq:ba-mass-morrey}
 \int_{B_r(z)}e^u\dd y
 \leq C_nr^{-1}\abs{B_{2r}}^{\frac{1}{2}}
       \left(\int_{B_{2r}(z)}\abs{\nabla u}^2\dd y\right)^{\frac{1}{2}}
 \leq C_n\sqrt\Lambda\,r^{n-2}.
\end{equation}
Thus the rescalings have bounded positive energy on every fixed ball.

We next check that their additive constants cannot escape.  Jensen's
inequality bounds $\fint_{B_1}u_{x,r}$ from above.  If along a sequence
these averages tended to $-\infty$, Poincar\'e's inequality and the uniform
gradient bound would imply $u_{x,r}\to-\infty$ in measure on $B_1$.
Lemma~\ref{lem:ba-entropy} yields
uniform integrability of $e^{u_{x,r}}$ on $B_1$ (apply it on $B_2$).
Hence $\int_{B_1}e^{u_{x,r}}\to0$.
Theorem~\ref{thm:ba-smallmass}, after a fixed rescaling, would make
$u_{x,r}$ smooth near zero, contradicting $x\in\Sing(u)$.
The averages on $B_1$ are therefore bounded below as well.  Poincar\'e's
inequality on larger concentric balls then bounds the $H^1$ norms on
every fixed ball.  Theorems~\ref{thm:ba-dichotomy} and~\ref{thm:ba-morrey},
with a diagonal subsequence, yield the stated strong convergence and
stationarity.  If $U$ were smooth near zero, its positive energy would be
small on a sufficiently small ball; strong convergence and
Theorem~\ref{thm:weak-epsilon} would yield the same contradiction.

The preceding bounds hold for all sufficiently small $r$, rather than
only for the selected sequence. The interior trace estimate on the annulus
$B_{\frac{3}{2}}\backslash\overline B_{\frac{1}{2}}$ yields
\[
 \abs{\int_{\partial B_1}u_{x,r}\dd\sigma}
 \leq C_n\norm{u_{x,r}}_{H^1(B_2)}\leq C.
\]
This uses only the $L^2$ bound for the $H^{\frac{1}{2}}$ trace. Together with
the energy bounds, it bounds $W_u(x,r)$ at almost every small $r$.
Proposition~\ref{prop:ba-weiss} implies the existence of a finite limit
$W_u(x,0)$.  For $0<a<b<\infty$, scaling its dissipation identity yields
\[
 \int_a^b t^{2-n}\int_{\partial B_t}
 \left(\partial_\nu u_{x,r_j}+\frac{2}{t}\right)^2\dd\sigma\dd t
 =W_u(x,r_jb)-W_u(x,r_ja)\to0.
\]
Strong $H^1$ convergence on annuli implies
$y\cdot\nabla U=-2$ almost everywhere. Integration on rays proves
\eqref{eq:ba-homogeneity}. The stronger angular conclusion
$\psi\in H^1(\mathbb S^{n-1})$ follows from annular slicing, not from
the trace estimate: choose a radius $t\in(\frac{1}{2},\frac{3}{2})$ at which the
angular slice of $U$ belongs to $H^1$ and its exponential is integrable,
and use $\psi=U(t\,\cdot)+2\log t$. Fubini theorem and the annular energy
bounds provide such radii. Substitution in the weak equation with
separated radial and angular test functions proves \eqref{eq:ba-angular}.
Testing that equation with one yields
\[
 \int_{\mathbb S^{n-1}}e^\psi\dd\sigma
 =2(n-2)\abs{\mathbb S^{n-1}},
 \quad \int_{B_1}e^U\dd y=2\abs{\mathbb S^{n-1}}.
\]
Strong $L^1(B_1)$ convergence proves this value for the mass along the
chosen subsequence.  Every sequence of radii admits such a subsequence
and all its tangent limits have the same mass, proving
\eqref{eq:ba-fixed-mass} for the full limit.
\end{proof}

\subsection{Dimension reduction}

We now apply the Federer dimension reduction argument to estimate
the Hausdorff dimension of the singular set. The two main
ingredients are the persistence of singular points under strong
convergence and the logarithmic homogeneity of tangent solutions.
Taking a further tangent at a non-zero singular point gives
an additional direction of translation invariance.
If the singular set had dimension greater than $n-3$, repeating
this procedure would produce a two-dimensional logarithmically
homogeneous solution with locally finite Dirichlet energy,
which is impossible.

\begin{thm}[Dimension reduction]\label{thm:ba-dimension}
Under the hypotheses of Theorem~\ref{thm:ba-tangents},
\[
 \dim_{\mathcal H}\Sing(u)\leq n-3.
\]
This bound is sharp, as shown by the codimension-three cylinders in
Section~\ref{sec:calibrations}.  If $n=3$, the singular set is locally
finite.  In this dimension, for each singular point $x$,
\[
 u(y)=-2\log\abs{y-x}+O(1)\quad\hbox{as }y\to x,
\]
and the mass density in \eqref{eq:ba-fixed-mass} equals $8\pi$.
\end{thm}
\begin{proof}
With the monotonicity formula (Proposition \ref{prop:ba-weiss}) in hand, the proof is a standard application of the Federer dimension reductuon principle. We need only to verify that singular points persist under the strong convergence
of Theorem~\ref{thm:ba-morrey}. Indeed, if $v_j\to v$ strongly in
$H^1_{\loc}$ and $e^{v_j}\to e^v$ strongly in $L^1_{\loc}$,
$x_j\in\Sing(v_j)$, and $x_j\to x$, then $x\in\Sing(v)$.
Indeed, regularity of $v$ at $x$, strong energy convergence on a small
ball, and Theorem~\ref{thm:weak-epsilon} contradict the presence of $x_j$
in its smaller concentric ball.  Consequently, for each compact $K$ and
open neighborhood $G$ of $\Sing(v)\cap K$,
$\Sing(v_j)\cap K\subset\subset G$ for all sufficiently large $j$.
\end{proof}

\begin{rem}
For $n>3$ the angular function in \eqref{eq:ba-angular} need not be smooth:
for $y\in\R^3$, $z\in\R^{n-3}$, the stationary Morrey-bounded solution
$U(y,z)=\log2-2\log\abs{y}$ has the singular plane $\{y=0\}$ and the angular
profile $\psi(\theta)=\log2-2\log\abs{\theta_y}$, which is singular where
$\theta_y=0$.  Thus even the Morrey assumption permits non-discrete
singular sets in higher dimensions. As in many other elliptic PDE problems, we do not claim the
uniqueness of tangent solutions.
\end{rem}

\subsection{The role of the Dirichlet energy bound}

Theorem~\ref{thm:ba-smallmass} shows that small mass implies
regularity when the Dirichlet energy has an a priori bound.
The required smallness of the mass, however, depends on this
bound. It is natural to ask whether the Dirichlet energy
assumption can be removed altogether. More precisely,
we have the following question.

\begin{q}\label{q:ba-puremass}
For every fixed $n\geq3$, do there exist $\delta_n>0$ and $C_n<\infty$
such that every stationary weak solution $u\in H^1(B_2)$ with
$e^u\in L^1(B_2)$ satisfying
\[
 \int_{B_2}e^u\dd x\leq\delta_n
\]
is smooth on $B_1$ and satisfies
$\sup_{B_1}e^u\leq C_n\int_{B_2}e^u$, without an a priori bound
for $\int_{B_2}\abs{\nabla u}^2$?
\end{q}

We believe the answer is negative, although at present we still cannot find a counterexample. The examples below
only show that small mass does not control the full Dirichlet energy; they
do not contradict either regularity assertion.

\begin{example}
\label{ex:ba-cosh}
For $a>0$ and $b>2$, the entire smooth function
\[
 u_{a,b}(x)=\log(2a^2)-2\log\cosh(a(x_1-b))
\]
is stationary and satisfies
\[
 -\Delta u_{a,b}=e^{u_{a,b}}
 =2a^2\operatorname{sech}^2(a(x_1-b)),\quad
 \nabla u_{a,b}=-2a\tanh(a(x_1-b))e_1.
\]
On $B_2$,
\[
 \int_{B_2}e^{u_{a,b}}
 \leq8a^2\abs{B_2}e^{-2a(b-2)}.
\]
With $a$ fixed and $b\to\infty$, the mass tends to zero, whereas
$\nabla u_{a,b}\to2ae_1$ uniformly on $B_2$. Thus even at bounded
Dirichlet energy there is no estimate
\[
 \int_{B_1}\abs{\nabla u}^2\leq C(n,\Lambda)\int_{B_2}e^u,
\]
nor an implication that the inner Dirichlet energy tends to zero with
the mass. Taking $a=j$ and $b=3$ instead makes the mass tend to zero
and the Dirichlet energy tend to infinity. All these solutions are
smooth and tend locally uniformly to $-\infty$ on $B_2$ in the indicated limits;
they do not answer Question~\ref{q:ba-puremass} negatively.
\end{example}

The dependence on the Dirichlet bound enters through the decomposition
$u=h+v$ and the identity
\[
 \int_{B_2}e^uv=\int_{B_2}\abs{\nabla v}^2
 \leq\int_{B_2}\abs{\nabla u}^2.
\]
Together with the $L\log L$ estimate, this controls the Hessians and the
tangent sequences in Lemma~\ref{lem:ba-defect}. The preceding estimates
do not yield these bounds from small mass alone.

\begin{example}
\label{ex:ba-bubbles}
Let $x=(y,z)\in\R^2\times\R^{n-2}$ and
\[
 U_\lambda(y,z)=\log\frac{8\lambda^2}{(1+\lambda^2\abs{y}^2)^2},
 \quad\lambda\to\infty.
\]
These are smooth stationary solutions. In the sense of locally finite
measures on $\R^n$,
\[
 e^{U_\lambda}\dd x\rightharpoonup
 8\pi\mathcal H^{n-2}\llcorner\{y=0\},
\]
while $U_\lambda\to-\infty$ uniformly on compact sets disjoint from
the plane and $U_\lambda(0,z)\to+\infty$.
Indeed, the two-dimensional source is an approximate identity of total
mass $8\pi$, and Fubini theorem  yields the displayed limit. Also, writing
$\omega_s=\abs{B_1^s}$,
\[
 \int_{B_1^n}\abs{\nabla U_\lambda}^2
 =32\pi\omega_s\log\lambda+O(1).
\]
This follows by integrating
$\frac{16\lambda^4\abs{y}^2}{(1+\lambda^2\abs{y}^2)^2}$ against the transverse
cross-section $\omega_s(1-\abs{y}^2)^{\frac{s}{2}}$ and splitting at
$\abs{y}=\lambda^{-1}$.
Hence, the gradient bound in Theorem~\ref{thm:ba-dichotomy} cannot
simply be deleted. This example generalizes the three-dimensional
one in \cite[Section~1]{DaLioHyder}. Its mass on a fixed ball centered
on the plane tends to a positive constant, so it too is not a
counterexample to Question \ref{q:ba-puremass} with an arbitrarily small invariant mass.
\end{example}

\section*{Acknowledgments}
Kelei Wang was supported by the National Key R\&D Program of China (No. 2022YFA1005602) and the National Natural Science Foundation of China (No. 12425108 and No. 12221001). Ke Wu was supported by the National Natural Science Foundation of China (No. 12401264) and Yunnan Revitalization Talent Support Program.  

The authors acknowledge the use of AI tools. All mathematical arguments and
proofs in the final manuscript were checked and written by the authors.

\bibliographystyle{plain}

\begin{thebibliography}{10}

\bibitem{Allard86}
W.~K. Allard.
\newblock An integrality theorem and a regularity theorem for surfaces whose
  first variation with respect to a parametric elliptic integrand is
  controlled.
\newblock Geometric measure theory and the calculus of variations, {Proc}.
  {Summer} {Inst}., {Arcata}/{Calif}. 1984, {Proc}. {Symp}. {Pure} {Math}. 44,
  1-28 (1986)., 1986.

\bibitem{BM91}
H.~Br{\'e}zis and F.~Merle.
\newblock Uniform estimates and blow-up behavior for solutions of {{\(-\Delta{}
  u =V(x) e^ u\)}} in two dimensions.
\newblock {\em Commun. Partial Differ. Equations}, 16(8-9):1223--1253, 1991.

\bibitem{ChaeWolf}
D.~Chae and J.~Wolf.
\newblock On the {Liouville} theorem for weak {Beltrami} flows.
\newblock {\em Nonlinearity}, 29(11):3417--3425, 2016.

\bibitem{DaLio}
F.~Da~Lio.
\newblock Partial regularity for stationary solutions to {Liouville}-type
  equation in dimension 3.
\newblock {\em Commun. Partial Differ. Equations}, 33(10):1890--1910, 2008.

\bibitem{DaLioHyder}
F.~Da~Lio and A.~Hyder.
\newblock Blow-up analysis of stationary solutions to a {Liouville}-type
  equation in {3D}.
\newblock {\em Commun. Contemp. Math.}, 27(2):30, 2025.
\newblock Id/No 2450020.

\bibitem{GrafakosKinnunen}
L.~Grafakos and J.~Kinnunen.
\newblock Sharp inequalities for maximal functions associated with general
  measures.
\newblock {\em Proc. R. Soc. Edinb., Sect. A, Math.}, 128(4):717--723, 1998.

\bibitem{Ledoux}
M.~Ledoux.
\newblock On improved {Sobolev} embedding theorems.
\newblock {\em Math. Res. Lett.}, 10(5-6):659--669, 2003.

\bibitem{Lin}
F.~Lin.
\newblock Gradient estimates and blow-up analysis for stationary harmonic maps.
\newblock {\em Ann. Math. (2)}, 149(3):785--829, 1999.

\bibitem{Lin01}
Fang-Hua Lin and Tristan Rivi\'{e}re.
\newblock A quantization property for static ginzburg-landau vortices.
\newblock {\em Comm. Pure Appl. Math.}, 54(2):206--228, 2001.

\bibitem{Mattila}
P.~Mattila.
\newblock {\em Geometry of sets and measures in {Euclidean} spaces. {Fractals}
  and rectifiability}, volume~44 of {\em Camb. Stud. Adv. Math.}
\newblock Cambridge: Univ. Press, 1995.

\bibitem{Moser}
R.~Moser.
\newblock Stationary measures and rectifiability.
\newblock {\em Calc. Var. Partial Differ. Equ.}, 17(4):357--368, 2003.

\bibitem{Stein}
Elias~M. Stein.
\newblock {\em Singular integrals and differentiability properties of
  functions}, volume~30 of {\em Princeton Math. Ser.}
\newblock Princeton University Press, Princeton, NJ, 1970.

\bibitem{TaoNotes}
Terence Tao.
\newblock Lecture notes 4 for 247a.
\newblock University of California, Los Angeles, 2006.

\bibitem{Wang12}
K.~Wang.
\newblock Partial regularity of stable solutions to the {Emden} equation.
\newblock {\em Calc. Var. Partial Differ. Equ.}, 44(3-4):601--610, 2012.

\end{thebibliography}

\end{document}